\RequirePackage{fix-cm}
\documentclass[smallextended]{svjour3}       
\smartqed  
\usepackage{graphicx}
\usepackage{newtxtext,newtxmath}         
\usepackage{booktabs}
\usepackage{array}
\newcolumntype{L}[1]{>{\raggedright\arraybackslash}p{#1}}
\usepackage{amsmath}
\usepackage{amsfonts}
\usepackage{algorithmic}
\usepackage[ruled,vlined]{algorithm2e}
\usepackage[colorlinks]{hyperref}
\usepackage{hyperref}
\hypersetup{colorlinks=true, pdfstartview=FitV, linkcolor=black, citecolor=black, plainpages=false,  urlcolor=black}
\usepackage{siunitx,multirow}
\usepackage{color,placeins}
\newcommand{\MAE}{Monge-Amp\'ere equation }
\newcommand{\bfx}{\textbf{x}}
\newcommand{\bfc}{\textbf{c}}
\newcommand{\bfw}{\textbf{w}}
\newcommand{\bfn}{\textbf{n}}
\newcommand{\bfv}{\textbf{v}}
 \journalname{3D MAE}
\begin{document}

\title{Trivariate Spline Collocation Methods for Numerical Solution to 3D Monge-Amp\`ere Equation}
\author{Ming-Jun Lai
		\and Jinsil Lee}
\institute{Ming-Jun Lai ,  Jinsil Lee  \at
              Department of Mathematics, University of Georgia, Athens, GA 30602 \\
              \\
              Ming-Jun Lai \at
              \email{mjlai@uga.edu}        
           \and
           Jinsil Lee \at
             \email{jl74942@uga.edu}       
}
\date{Received: date / Accepted: date}
\maketitle
\begin{abstract}
We use trivariate spline functions for the numerical solution of the Dirichlet problem of the 3D 
elliptic Monge-Amp\'ere equation.  Mainly we use the spline collocation method introduced in [SIAM J. Numerical Analysis, 2405-2434,2022] to numerically solve iterative Poisson equations and use an averaged algorithm to 
ensure the convergence of the iterations. We shall also establish the rate of convergence under 
a sufficient condition and provide some numerical evidence to show the numerical rates. Then we 
present many computational results to demonstrate that this approach works very well. In 
particular, we tested many known convex solutions as well as nonconvex solutions over convex and 
nonconvex domains and compared them with several existing numerical methods to show the 
efficiency and effectiveness of our approach.

\end{abstract}

\section{Introduction}
\label{intro}
We are interested in numerically solving the Monge-Amp\'ere equation with 
Dirichlet boundary condition: 
\begin{eqnarray}
\label{MA}
	\det(D^{2}u(\bfx))&=f(\bfx), ~~\text{in} ~ \Omega \subset \mathbb{R}^3\\
	u(\bfx)&=g(\bfx), ~~\text{on} ~\partial \Omega,
\end{eqnarray}
where $\bfx=(x,y,z)$ has $3$ independent variables in a bounded domain 
$\Omega \subset \mathbb{R}^3$ and $D^2u$ is the Hessian of the function $u$, more precisely, 
\begin{equation}
	\label{3D}
	\det(D^2 u)=u_{xx} u_{yy} u_{zz} + 2u_{xy} u_{yz} u_{xz}-u_{xx}(u_{yz})^2- u_{yy} (u_{xz})^2-u_{zz} (u_{xy})^2.
\end{equation} 
This is a first step toward to solve the fully nonlinear Monge-Amp\'ere equation 
\begin{align}
	\label{MAE}
	\det(D^2u(\bfx))&=f(\bfx)/g(\nabla u(\bfx)), ~~\bfx~\text{in} ~ \Omega \subset \mathbb{R}^3\\
	\nabla u(\bfx)|_{\partial \Omega} &=\partial W, 
\end{align} 
where the boundary condition is called  the oblique boundary condition. Such a partial 
differential equation arises from the optimal transportation problem 
(cf. e.g. \cite{E98} and  \cite{V03}).  
More specifically, given a density function $f(\bfx)$ on the domain $\Omega$ and another density function $g(\bfw)$ on a separate domain $W$, the goal is to find the optimal plan $T$ which transports $f$ over $\Omega$ to $g$ over $W$ 
under the cost functional $c(\bfx, \bfw)= \dfrac{1}{2}\|\bfx- \bfw\|^2$, 
with $\int_\Omega f(\bfx)d\bfx=\int_W g(\bfw)d\bfw$.  It is Y. Brenier who discovered a 
characterization of the optimal transportation problem.
\begin{theorem}(Brenier, 1988\cite{BB00}) 
	Suppose that the transport cost is the quadratic Euclidean distance, $c(x, y) = \frac{1}{2}
	\|x - y\|^2$ and suppose that $W$ is a convex domain. Then
	there exists a convex function $u : \Omega \mapsto \mathbb{R}$ satisfying the 
	Monge-Ampere equation (\ref{MAE}), unique up to a constant, such that the gradient map $m=\nabla u$ is the unique optimal transport map satisfying the oblique boundary condition 
	$\nabla u|_{\partial \Omega} = \partial W$.  
\end{theorem}

Although it is hard to determine the oblique boundary condition mentioned above, 
once we specify a map from the boundary of
$\Omega$ to the boundary of $W$, the problem (\ref{MAE}) becomes  
a Neumann boundary problem of the \MAE.  In particular, 
if $u$ is $C^2$ function whose gradient $\nabla u$ transforms $\Omega$ onto $W$,   
we can move the density $f(\bfx)$ at $\bfx \in \Omega$ to the location $\nabla u(\bfx)\in W$ to 
become the density $g(\nabla u(\bfx))$.
Such a problem is called the free movement problem which will be addressed 
at the end of this paper. 

Instead of considering the Neumann or oblique boundary value problem, this paper will focus on the Monge-Amp\'ere equation with a Dirichlet boundary condition.
 Note that this PDE has been studied for many years.  
In addition to the mathematical community, the \MAE has also been broadly studied in many 
applied fields such as elasticity, geometric optics, and image processing. See \cite{B06} 
and \cite{MY01}. 
Today such free-form optics are important in illumination applications. For example, they are 
used in the automotive industry for the construction of headlights that use the full light 
emitted by the lamp to illuminate the road but at the same time do not glare oncoming traffic 
\cite{ZNC11}. There are multiple ways to solve this inverse reflector problem; brute-force 
approaches, methods of supporting ellipsoids, simultaneous multiple surfaces approach, and 
Monge-Amp\'ere approaches.  Also, the \MAE finds applications in finance, seismic wave 
propagation, geostrophic flows, in differential geometry as explained in \cite{CGG18}.  
In this paper, we shall explain a spline based collocation method to solve the nonlinear 
PDE (\ref{MA}).

Let us begin recalling some existence, uniqueness, and regularity property of 
the \MAE (\ref{MA}). When $f, g$ are sufficiently smooth, the solution of (\ref{MA}) is very smooth explained in the following 

\begin{theorem}(Theorem 1 in \cite{CNS84})
Suppose that a bounded domain $\Omega\in \mathbb{R}^n$ is strictly convex, where $n\ge 2$. 
For any strictly positive right-hand side $f\in C^\infty(\overline{\Omega})$ 
with the boundary condition $g$ which has an extension $g\in C^\infty(\overline{\Omega})$, 
there exists a unique strictly convex solution $u$ is in $C^\infty(\Omega)$ satisfying (\ref{MA}).
\end{theorem} 

There are several weaker versions of the existence results with regularity properties in the literature. For example,  
\begin{theorem}(Figalli, 2017 \cite{F17})
Let $\Omega$ be a uniformly convex domain, $k\geq 2, \alpha \in (0,1),$ and assume that 
$\partial \Omega$ is of class $C^{k+2,\alpha}$. 
Let $f\in C^{k,\alpha}(\bar{\Omega})$ with 	$f\geq c_0>0.$ 
Then for any $g\in C^{k+2,\alpha}(\partial \Omega),$ there exists a unique 
	solution $u\in C^{k+2,\alpha}(\bar{\Omega})$ to the Dirichlet problem \eqref{MA}.
\end{theorem}

In \cite{A13P}, Awanou introduced another weaker version of the existence theorem:
\begin{theorem}(Awanou, 2013\cite{A13P})
Let $\Omega$ be a uniformly convex domain in $\mathbb{R}^n$ with boundary in $C^3$. 
Suppose $g\in C^3(\bar{\Omega}), \inf f >0$, and $f\in C^\alpha(\bar{\Omega})$ 
for some $\alpha \in (0,1)$. 
Then \eqref{MA} has a convex solution $u$ which satisfies the a priori estimate
\begin{equation*}
\|u\|_{C^{2,\alpha}(\bar{\Omega})}\le C,
\end{equation*}
where $C$ depends only on $n\ge 2, \alpha, \inf f, \Omega, \|f\|_{C^{\alpha}(\bar{\Omega})}$ 
and $\|g\|_{C^{3}}.$
\end{theorem}

In general, there are at least three different notions of solutions which have been 
studied in the literature besides the 
classic solution: one is called Aleksandrov solution,  
another one is viscosity solution,  and the next one is Brainer's solution, according to the 
monograph by Villani, 2003, see page 129 in \cite{V03}. 
The theory for the \MAE is deep (cf. \cite{CC95}, \cite{E98}, \cite{V03} and \cite{V08}). 
In particular, the regularity of
the solution has been extensively studied (cf. e.g. \cite{C90},  \cite{W96},  \cite{CLW21}). 
In a landmark paper \cite{C90}, Cafferelli showed that the solution of the \MAE has 
an interior regularity over $\Omega'\subset \Omega$, i.e. $u\in H^{2, p}(\Omega')$ for any open set $\Omega'$ inside $\Omega$. 
Furthermore,   the solution has $H^2$ regularity over the entire domain, as established in \cite{W96}:
\begin{theorem}(Wang, 1996\cite{W96})
Let $\Omega$ be a strictly convex domain in $\mathbb{R}^n.$ If $\partial \Omega$ and $g$ in the equation (2) are $C^3$ smooth, and $f(x)\in C^{1,1}(\bar{\Omega}),$ then the solution $u\in C^{2+\alpha}(\bar{\Omega}).$
\end{theorem}
Due to these regularity results, we can use $C^2$ smooth trivariate splines to approximate the solution $u$ 
under the conditions $f\in C^{1,1}(\bar{\Omega})$, $g\in C^3(\partial \Omega)$ and $\Omega$ being a strictly convex domain. 
In our computation, we are able to solve the Monge-Amp\'ere equation over domains with uniform 
positive reach (cf. \cite{GL20}) which include strictly convex domains as a special case. Additionally, we can use our method to experiment with the solution (\ref{MA}) even when $f$ is not in $C^{1,1}(\bar{\Omega})$.  

The numerical solution of the Monge-Ampere equation (MAE) is an active area of research, with many researchers developing different numerical methods and analyzing their theoretical convergence. As mentioned in \cite{BFO10}, the MAE poses several challenges for numerical solutions. The first challenge is that the equation is fully nonlinear, which means that geometric solutions or viscosity solutions must be used as weak solutions. The second challenge is the convex constraint, as the equation might not have a unique solution without it. 

The popular finite element method is not directly applicable because of the involvement of the Hessian of the 
solution. This restricts the use of the Finite Element Method (FEM) or general 
Galerkin projection methods, discontinuous Galerkin method or continuous Galerkin method.  
However, there are several remedy approaches based on the finite element method such as a 
mixed finite element method, vanishing moment method, etc.  See \cite{A13}, \cite{A14}, 
\cite{AL14} and \cite{FN09a}.  

Besides of finite element type methods, there are many finite difference methods, as seen in  \cite{BFO10}, \cite{A16}, \cite{LFOX16}, \cite{VNNP20}, \cite{LG21}. 
Moreover, many interesting approaches are based on the classic finite difference method as demonstrated in \cite{BS19}, \cite{LX20}, \cite{VNNP20}. 
However, these methods have a weakness: they do not have analytic form of solution over 
the entire domain.  
In addition, we can find time marching methods in \cite{A15}, \cite{A13P}, and 
least squares relaxation methods in \cite{DG03}, \cite{DG04}, \cite{CGG18}.  

Let us be more precise on the numerical methods mentioned above. 
The least square notion of the solution was proposed and studied in \cite{DG03}, \cite{DG04}, 
and \cite{CGG18}. Especially, this least square approach using a relaxation algorithm of the 
Gauss-Seidel-type iterations to decouple differential operators 
in \cite{CGG18}.  The approximation relies on mixed low order finite element methods with 
regularization techniques. Several 3D examples were demonstrated to show the performance 
of this method.   In this paper, we will compare the numerical results from our method to those to in \cite{CGG18} to show that our method produces more accurate results. 
These comparisons will be presented in the last section. 

In \cite{A15}, a time marching approach is used to solve the \MAE. 
Given $\nu>0,$ the researcher considered the sequence of iterates 
\begin{equation}
	\label{A15}
	-\nu \Delta u_{k+1}=-\nu \Delta u_k+\text{det}D^2 u_k-f, ~u_{k+1}=g ~~\text{on} ~\partial \Omega.
\end{equation}
He used the discrete version of Newton's method in the vanishing moment methodology. And he 
showed the convergence of the iterative method for solving the nonlinear system. 
We shall also compare  his numerical results with our results
in the last section to  show that our proposed method is also more accurate.    

In \cite{VNNP20}, the researchers introduced the meshless Generalized Finite Difference 
Method (GFDM) in both 2D and 3D settings. They tested several examples using the Cascadic 
iterative algorithm over convex and non-convex domains. 
We will compare our proposed method with the results from the Cascadic iterative algorithm in the last section to demonstrate that our method is also better.  

We now describe our numerical method to solve (\ref{MA}) by using trivariate spline 
functions  over a tetrahedralization of $\Omega$. See \cite{LS07}, 
\cite{ALW06}, \cite{S15}, \cite{LL21} for theoretical properties and numerical implementation
of bivariate/triavariate spline functions.  In addition, there are several dissertations written 
to explain how to implement and how to use multivariate splines for the numerical solution of 
Helmholtz equations, Maxwell equations and 3D surface reconstruction. See \cite{A03}, 
\cite{M19} and \cite{X19}.  
There are several reasons why we use trivariate splines for the numerical solution of the \MAE. 
One is that we can use trivariate splines with smoothness $r\ge 2$ to approximate the solution
$u$ of the \MAE over an arbitrary convex polyhedral domain.  Due to the $C^2$ smoothness of spline function, we can calculate the Hessian of the solution, so that we simply use the collocation method instead of the weak formulations in  \cite{A13}, \cite{A14}, etc..
Many researchers adopted the iterative algorithm called the fixed point algorithm introduced in \cite{BFO10}:
\begin{align}
	\label{2DMA1}
	\Delta u_{k+1}= ((\Delta u_k)^n+a (f- \det D^2 u_k))^{\frac{1}{n}}
\end{align}
along with the prescribed Dirichlet boundary conditions with $a=2$ and $n=2$. The researchers
in \cite{BFO10} explained that this is a fixed point method as the true solution $u$ satisfies
(\ref{2DMA1}) trivially. 

In \cite{A15}, this iterative algorithm is generalized to the 3D setting with $n\ge 3$ for 
various $a>0$. In particular, the researcher in \cite{A15} explained that the iteration 
(\ref{2DMA1}) is well-defined for $a\le n^n$ as $\det (D^2 u_k)\le \displaystyle \frac{1}{n^2}
(\Delta u)^n$.   Numerical results in \cite{A15} are demonstrated in 
the framework of the spline element method with $a=2$ for the 2D case and $a=9$ for the 3D case. 

In this paper, we shall use the following iterative method:
\begin{align}
	\label{2DMA}
	\Delta u_{k+1}= ((\Delta u_k)^n+n^n (f-\det D^2 u_k))^{\frac{1}{n}}
\end{align}
to handle the nonlinearity of the \MAE where $n=3$.

However, another requirement of the solution of the \MAE is that $u$ must be convex in order for the equation to be elliptic. Without this constraint, the equation does not have a unique solution. (For example, taking 
boundary data $g = 0$, if $u$ is a solution, then $-u$ is also a solution in $\mathbb{R}^2$.) Many numerical methods mentioned above failed to enforce this convexity constraint. The convexity of $u$ is equivalent to the positive definiteness of the Hessian matrix $D^2u$. In terms of the eigenvalues $\lambda_1\ge \lambda_2\ge \lambda_3$ of $D^2u$, we will 
ensure that three eigenvalues  
$\lambda_1(k)\ge \lambda_2(k)\ge\lambda_3(k)$ of the $k$th iteration $u_k$ 
in a spline space satisfy  
$\lambda_1(k)+ \lambda_2(k)+\lambda_3(k) \ge 0$ as well as $\lambda_1(k) 
\lambda_2(k)\lambda_3(k)>0$, although 
they are not enough to ensure the convexity of $k$th spline solution $u_k$.  

This paper is organized as follows. 
In Section \ref{sec:1}, we first explain trivariate splines, domains with uniformly positive reach, and the spline collocation method for the 
Poisson equation which is the same as the one  discussed in \cite{LL21}. 
 In Section \ref{sec:2}, we introduce the spline collocation method for 
the \MAE and its average algorithm, and establish 
three different versions of convergence results. 
Finally, in the last section \ref{sec:3}, we present numerical 
results for several 3D examples of smooth and convex solutions, as well as 
nonsmooth convex solution over  convex and nonconvex bounded domains to  
demonstrate the effectiveness of our proposed method.  
We compare our results with those of several existing numerical methods to show  the accuracy and efficiency of our method. 
Finally, we shall present some examples
for the free movement in 2D and 3D settings to show how the density from one place 
is moved to another place. This will demonstrate further that our proposed method is versatile 
enough.   
\section{Preliminaries}
\label{sec:1}
\subsection{Trivariate Splines}
\label{sec:1.1}
Let us quickly summarize the essentials of trivariate splines in this section. 
Given a tetrahedron $T$, we write $|T|$ for the length of its longest edge, and $\rho_T$ for the radius of the largest inscribed ball in $T$. 
For any polygonal domain $\Omega\subset \mathbb{R}^3$, 
let $\triangle:=\{T_1,\cdots, T_n\}$ be a tetrahedralization of 
$\Omega$ which is a collection of tetrahedra and $\mathcal{V}$ be the set of vertices of $\triangle$. 
{We called a tetrahedralization as a quasi-uniform tetrahedralization if all tetrahera $T$ in $\triangle$ have comparable sizes in the sense that 
	\begin{equation*}
	\frac{|T|}{\rho_T}\le C<\infty, ~~~~\text{all tetrahera} ~ T\in \triangle ,
		\end{equation*}
	where $\rho_T$ is the inradius of $T$. Let $|\triangle|$ be the length of the longest edge in $\triangle.$} 

Next for a tetrahedron  $T=(\bfv_1, \bfv_2, \bfv_3, \bfv_4) \in \triangle,$ 
we define the barycentric coordinates $(b_1, b_2,b_3, b_4)$ 
of a point $(x,y,z)\in \Omega$ as the solution to the following system of equations 
\begin{eqnarray*}
	b_1+b_2+b_3+b_4=1\\
	b_1 v_{1,x}+b_2 v_{2,x}+b_3 v_{3,x}+ b_4 v_{4,x} =x\\
	b_1 v_{1,y}+b_2 v_{2,y}+b_3 v_{3,y}+ b_4 v_{4,y} =y\\
	b_1 v_{1,z}+b_2 v_{2,z}+b_3 v_{3,z}+ b_4 v_{4,z} =z,
\end{eqnarray*}
where the vertices $\bfv_i=(v_{i,x}, v_{i,y}, v_{i,z})$ for $i=1,2,3, 4$. $b_1, \cdots, b_4$ are nonnegative if $(x,y,z)\in T.$ Next we use the barycentric coordinates to define the Bernstein polynomials of degree $D$:
\begin{eqnarray*}
	B^T_{i,j,k,\ell}(x,y,z):=\frac{{D}!}{i!j!k!\ell!}b_1^i b_2^j b_3^k b_4^\ell, ~i+j+k+\ell=D,
\end{eqnarray*}
which form a basis for the space $\mathcal{P}_D$ of polynomials of total degree $D$. 
Therefore, we can represent all $s\in \mathcal{P}_D$ in B-form:
\begin{eqnarray*}
	s|_T=\sum_{i+j+k+\ell=D}c_{ijk\ell}^T B^T_{ijk\ell}, \forall T\in \triangle,
\end{eqnarray*}
where the B-coefficients $c^T_{i,j,k,\ell}$ are uniquely determined by $s$. 
Let $\bfc=\{c^T_{ijk\ell}, i+j+k+\ell=D, T\in \triangle\}$
be the coefficient vector associated with spline function $s$. 

Moreover, for given $T=(\bfv_1,\bfv_2, \bfv_3, \bfv_4)\in \triangle$, 
we define the associated set of domain points to be 
\begin{equation}
	\label{domainpoints}
	\mathcal{D}_{D,T}:= \{\frac{i\bfv_{1}+j\bfv_{2}+k\bfv_{3}+\ell \bfv_4}{D} \}_{i+j+k=D}.
\end{equation}
Let $\mathcal{D}_{D,\triangle} = \cup_{T\in \triangle} \mathcal{D}_{D,T}$ be the domain 
points of tetrahedral $\triangle$ and degree $D$.  

We use the discontinuous spline space $S^{-1}_D(\triangle):= 
\{s|_{T} \in \mathcal{P}_D, T\in \triangle\}$ as a base. 
Then we add the smoothness conditions to define the space 
$\mathcal{S}^r_D:=C^r(\Omega)\cap S^{-1}_D(\triangle).$ 
The smoothness conditions are explained in \cite{LS07}. Indeed, see Theorem  15.31 in \cite{LS07}. 
We use $C^r$ smooth spline functions in $H^2(\Omega)$ with $r\geq 1$ 
and the degree $D$ of splines sufficiently large satisfying $D\geq 3r+2$ in $\mathbb{R}^2$ and 
$D\geq 6r+3$ in $\mathbb{R}^3$. 
And we get the following Lemma in \cite{LS07}
\begin{lemma}\label{spline_LS}
 For all $u\in W^{m+1,p}(\Omega)$ for some $0\le m\le D$ and $1\le p\le \infty$, there exists a quasi-interpolatory spline $s_u\in \mathcal{S}^r_D(\triangle)$ such that 
	$$\|D^\alpha (u-s_u)\|_{L^p(\Omega)}\le C |\triangle|^{m+1-|\alpha|}|u|_{m+1,p,\Omega},$$
	for all $0\le |\alpha|\le m,$ where $|\cdot|_{m+1,p,\Omega}$ is a semi-norm, $C$ is a positive constant independent of $u$ and $|\triangle|$ but is dependent on the geometry of $\triangle$.
	\end{lemma}
In addition, we shall use Markov inequality(cf. \cite{LS07}):
\begin{equation}
\|\nabla s\|_{\infty, \Omega}\le \frac{C}{|\triangle|} \|s\|_{\infty, \Omega}, \quad \forall s\in \mathcal{S}^r_D(\triangle)
\end{equation}
for a positive constant $C$ independent of $s$ and the size $|\triangle|$ of tetrahedralization $\triangle$.

\subsection{Domains with Uniformly Positive Reach} 
Let us  recall a concept on domains of interest explained in \cite{GL20}. 
\begin{definition}
	Let $K\subseteq \mathbb{R}^n$ be a non-empty set. Let $r_K$ be the supremum of the number $r$ 
	such that every points in 
	$$P=\{x\in \mathbb{R}^n: \text{dist}(x,K)<r\}$$
	has a unique projection in $K.$ The set $K$ is said to have a positive reach if $r_K>0.$ 
\end{definition}
A domain with $C^2$ boundary has a positive reach. Sets of positive reach are much more general than convex sets.
Let $B(0,\epsilon)$ be the closed ball centering at 0 with radius $\epsilon>0,$ and let $K^c$ stand for the complement of the set $K\in \mathbb{R}^n.$ For any $\epsilon >0,$ the set
$$E_\epsilon (K):= (K^c+B(0,\epsilon))^c \subseteq K$$
is called an $\epsilon $-erosion of $K.$ Next we recall the following definition from  \cite{GL20}. 
\begin{definition}
A set $K\subseteq \mathbb{R}^n$ is said to have a uniformly positive reach $r_0$ if there exists 
some $\epsilon_0>0$ such that for all $\epsilon \in [0,\epsilon_0], E_\epsilon (K)$ has a 
positive reach at least $r_0.$ 
\end{definition}
Many examples of domains with positive reach can be found in \cite{GL20}. 
And we have the following property about these domains 
\begin{lemma}
If $\Omega \subset \mathbb{R}^n$ is of positive reach $r_0$, then for any $0<\epsilon<r_0$, the 
boundary of $\Omega_\epsilon :=\Omega+B(0,\epsilon)$ containing $\Omega$ is of $C^{1,1}$. 
Furthermore, $\Omega_\epsilon$ has a positive reach $\geq r_0-\epsilon.$
\end{lemma}

In \cite{GL20}, Gao and Lai proved the following regularity theorem 
\begin{theorem}\label{thm:regularity} Let $\Omega$ be a bounded domain. Suppose the closure of 
$\Omega$ is of uniformly positive reach $r_\Omega$. For any  $f\in L^2(\Omega),$ let $u\in 
H^1_0(\Omega)$ be the unique weak solution of the Dirichlet problem:
\begin{align*}
	\begin{cases}
			-\Delta u&=f~~in~\Omega\\
			u&=0~~on~\partial \Omega
		\end{cases}
\end{align*}
Then $u\in H^{2}(\Omega)$ in the sense that
$$\sum_{i,j=1}^n \int_\Omega (\frac{\partial^2 u}{\partial x_i \partial x_j})^2\le C_0\int_\Omega f^2 dx$$
for a positive constant $C_0$ depending only on $r_\Omega$.
\end{theorem}

\subsection{A Collocation Method for the Poisson Equation}
\label{sec:1.2}
For convenience, let us start with the Poisson equation
\begin{eqnarray}
	\label{Poisson}
	\Delta u(\bfx) &=& f(\bfx) \quad \forall \bfx\in \Omega \subset \mathbb{R}^3\cr
	u(\bfx) &=& g(\bfx), \quad \forall \bfx\in \partial \Omega
\end{eqnarray}
For given $\triangle$, let it be a tetrahedral partition of $\Omega$, we choose a set of domain points 
$\{ \xi_i\}_{i=1,\cdots, N}$ explained in the previous section as collocation points 
and let $s=\sum_{t\in \triangle}\sum_{|\alpha|=D}c^{t}_{\alpha} \mathcal{B}^t_\alpha$ in $S^r_D(\triangle)$ 
with the coefficient vector $\textbf{c}$ of $s$. Then we want to 
find the coefficient vector $ \textbf{c}$ of spline function 
satisfying the standard Poisson equation (\ref{Poisson}) at those collocation points
\begin{equation}
	\label{Poisson2}
	\begin{cases} 
		\Delta s(\xi_i)  & = f(\xi_i), \quad \xi_i\in \Omega \subset \mathbb{R}^n, \cr 
		s (\xi_i) & = g(\xi_i),  \quad \xi_i\in \partial \Omega,
	\end{cases}
\end{equation}
where $\{ \xi_i\}_{i=1,\cdots, N} \in \mathcal{D}_{D',\triangle}$  are the domain points of $\triangle$ of degree $D'>0$
as explained in (\ref{domainpoints}) in the previous section, where $D'$ will be different from $D$($D'>D$). 

Using these points, we let $K$ be the following matrix:
\begin{align*}
	K&:= \begin{bmatrix} \Delta (B^t_{i,j,k,l})(\xi_i) \end{bmatrix}.
\end{align*}

In general, the spline $s$ with coefficients in $\bf{c}$ is a discontinuous function.    
In order to make $s\in \mathcal{S}^r_D(\triangle)$, 
its coefficient vector $\textbf{c}$ must satisfy the constraints $H\textbf{c}=0$ for the smoothness conditions that 
the $\mathcal{S}^r_D(\triangle)$  functions  possess (cf. \cite{LS07}).  
Based on the smoothness conditions (cf. Theorem 2.28 or Theorem 15.38 in \cite{LS07}), 
we can construct matrices $H$ for the $C^r$ smoothness  conditions.
Then, our collocation method is to find ${\bf c}$ which solves the following constrained minimizations: 
\begin{align}
	\label{min1ma}
	\min_{\bf c} J({\bf c})=\frac{1}{2}(\|B{\bf c} -  G \|^2+\|H{\bf c} \|^2)\\ \text{subject to } 
	K\bf c= {\bf f}
\end{align}
where $B, G$ are from the boundary condition and $H$ is from the smoothness condition. Note that we need to 
justify that the minimization has a solution. In general,  we do not know if $ K{\bf c} = {\bf f}$  has a solution or 
not. However, we can show that a neighborhood of  $K \textbf{c}= \bf f$, i.e. 
\begin{equation}\label{neighbor}
	\mathbb{N}=\{\textbf{c}: \|K{\bf c}- {\bf f}\|\le \epsilon\}
\end{equation}
is not empty for an appropriate $\epsilon>0$ when $|\triangle|$ is small enough. 
 Indeed, let us use spline
approximation theorem (cf. \cite{LS07}) to have 
\begin{lemma}
	\label{lem3} 
	Suppose that $\Omega$ is a polygonal domain. 
	Suppose that $u\in H^3(\Omega)$. Then there exists a spline function $u_s\in S^r_D(\triangle)$ and a 
	positive constant $\hat{C}$ depending 	on $D\ge 1$ and $D'> D$ such that 
	\begin{eqnarray*}
		\|\Delta u(x,y,z)-\Delta u_s(x,y,z)\|_{L^\infty(\Omega)}\le \epsilon_1\hat{C}
	\end{eqnarray*}
where $\epsilon_1$ is a given tolerance and $\hat{C}$ depends on $\Omega, D, D'$.
\end{lemma}

We thus consider a nearby problem of the minimization (\ref{min1ma}), which is:
\begin{align}
	\label{min2ma}
	\min_{\bf c} &\|B{\bf c} -G\|^2+ \|H{\bf c}\|^2,\\
	&\hbox{subject~ to ~} \|K{\bf c}- {\bf f}\|_{L^\infty} \le \epsilon_1.
\end{align}
It is easy to see that the minimizer of (\ref{min2ma})  clearly approximates the minimizer of (\ref{min1ma})
if $\epsilon_1\ll 1$. 
As the new minimization problem is convex and the feasible set is also convex, the minimization (\ref{min2ma}) will have
a unique solution if the feasible set is non-empty.

We may assume that our numerical solution $u_s$ approximates $u$ on $\partial \Omega$ very well 
in the sense that $\| u-u_s\| \le C\epsilon_2$ for a positive constant $C$. Denote $\|u\|_L:=\|\Delta u\|_{L^2(\Omega)}$.
In \cite{LL21}, Lai and Lee proved the following theorems
\begin{theorem}
	\label{mainthm2ma}
	Suppose $f$ and $g$ are continuous over bounded domain $\Omega \subseteq \mathbb{R}^d$ for 
	$d=2$ or $d=3$. Suppose that $u\in H^3(\Omega)$. When $\Omega$ is a domain with uniform positive reach, 
	we have the following inequality
	\begin{align*}
		\|u-u_s\|_{L^2(\Omega)} \le C\|u-u_s \|_{L}, \|\nabla (u-u_s)\|_{L^2(\Omega)} \le C\|u-u_s \|_{L}
	\end{align*}
	and 
	\begin{align*}
		\sum_{ i+j=2}\| \frac{\partial^2}{\partial x^i \partial y^j}u\|_{L^2(\Omega)}\le  C \|u-u_s \|_{L}
	\end{align*}
	for a positive constant $C$ depending on $\Omega$.
\end{theorem}
And we can obtain the better convergence results if we assume that $|u-u_s|_{\partial \Omega}=0$, as shown in the following theorem:
\begin{theorem}
	\label{mjlai05122021} Suppose that $|u-u_s|_{\partial \Omega}=0$.  
	Under the assumptions in Theorem~\ref{mainthm2ma},  we have the following inequality
	\begin{align*}
		\|u-u_s\|_{L^2(\Omega)} \le C|\triangle|^2 (\|u-u_s \|_{L}) \hbox{ and }
		\|\nabla (u-u_s)\|_{L^2(\Omega)} \le C|\triangle| (\|u-u_s \|_{L})
	\end{align*}
	for a positive constant $C$, where $|\triangle|$ is the
	size of the underlying tetrahedral $\triangle$. 
\end{theorem}
\begin{proof}
We use the same arguments as in \cite{LL21} to establish a proof. \end{proof}
\section{Our Proposed Algorithms and Their Convergence Analysis}
\label{sec:2}
\subsection{A Spline Based Collocation Method for Monge-Amp\`ere Equation}
\label{sec:2.1}
For convenience, let us explain our spline based collocation method for the 3D \MAE first. 
In the 3D setting, we will solve the following iterative equations as in \cite{A15}
\begin{equation}
	\label{Iteration2D}
	\Delta u_{k+1}=\sqrt[3]{(\Delta u_k )^3+a(f-\text{det} (D^2 u_k))}, \quad k=0, 1, \cdots, . 
\end{equation}
with an initial $u_0$ by solving   
	\begin{equation*}
\Delta u_{0}=\sqrt[3]{27f}
	\end{equation*}
together with the given boundary condition. Let us make two quick remarks. 
The initial $u_0$ so chosen is based on the following 
assumption. Writing $\lambda_i, i=1,2,3,$ to be the eigenvalues of $\text{det} (D^2 u)$, 
the Monge-Amp\'ere equation reads $\lambda_1 \lambda_2 \lambda_3 =f$. 
If these eigenvalues are close to each other, e.g.  all are equal to $\lambda$, we have 
$\lambda^3=f$ and thus $\lambda= \sqrt[3]{f}$. Since $\Delta u=\lambda_1+\lambda_2+\lambda_3$, 
we get $\Delta u=3\lambda=3 \sqrt[3]{f}=\sqrt[3]{27f}.$  However, when these eigenvalues are
quite different, such a choice of initial $u_0$ may not be a good one. We shall explain our 
approach later in this section. 

Next about the parameter $a$ in \eqref{Iteration2D}, we have tested different numbers 
for $a$. Figure \ref{MA:diffa} shows that we can get more accurate results 
when using $a=27.$  We will use $a=27$ in the rest of the paper.  
\begin{figure}[htpb]
	\begin{tabular}{cc}
	\includegraphics[width=.5\linewidth,height=0.5\linewidth]{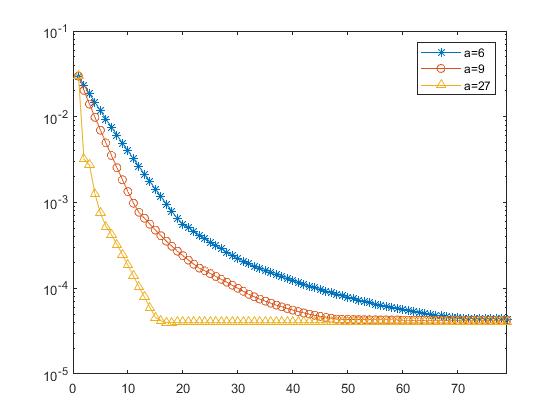}
	&
	\includegraphics[width=.5\linewidth, height=0.5\linewidth]{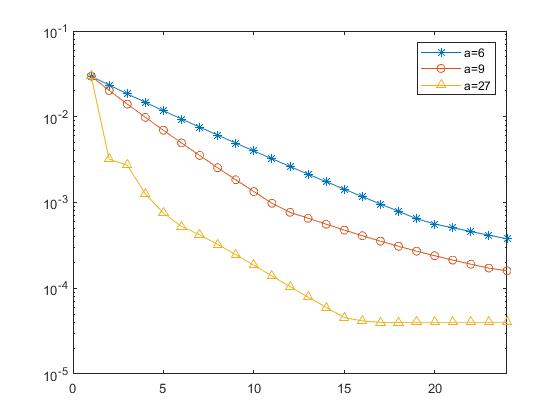}
	\cr
\end{tabular}
	\caption{$\log(\|u-u^{(k)}_s\|_\infty)$($y$-axis) for $u^{3d3}$(left) and $u^{3d5}$(right) for each iteration($x$-axis) with different $a=6, 9, 27$}\label{MA:diffa}
\end{figure}

For simplicity, we use the Dirichlet boundary condition
$u|_{\partial \Omega}= g$ to explain our numerical method.

\subsection{Two Computational Algorithms}
\label{sec:2.2}
We shall present two computational algorithms for numerical solution of the \MAE. 
The first one is a standard approach which 
has been used by many researchers  based on finite different discretization and 
finite element discretization in the literature. 
We shall use multivariate spline functions to discretize the
function space $H^2(\Omega)$ to demonstrate the efficiency and effectiveness of our
multivariate spline approach as well as show that our numerical results are better than many 
methods in the literature.  See our computational results in the next section. 

\begin{algorithm}\caption{An Iterative Poisson Equation Algorithm}
	\label{alg1.1}
	Start with an initial solution $u_0$  by solving the following Poisson equation 
	using our collocation method discussed in the previous section:  
	\begin{equation}
		\label{initialguess}
		\Delta u_{0}=\sqrt[3]{27f}
	\end{equation}
	and $u_0= g$ on the boundary $\partial \Omega$.  
	We then iteratively solve the Poisson equation 
	\begin{equation}
		\Delta u_{k+1}=\sqrt[3]{(\Delta u_k )^3+27(f-\text{det} (D^2 u_k))}, \quad  ~k=0, 1, \cdots.  
	\end{equation}
	That is, we find $u_{k+1}\in S^r_D(\triangle)$ satisfying the following equations approximately: 
	\begin{equation}
		\label{PDE2ma}
		\begin{cases} \Delta u_{k+1}(\xi_i)  & = \sqrt[3]{(\Delta u_k(\xi_i) )^3+27(f(\xi_i)-det (D^2 u_k(\xi_i))) }
			\quad \xi_i\in \Omega \subset \mathbb{R}^3, \cr 
			u_{k+1} (\xi_i) & = g(\xi_i),  \quad \xi_i\in \partial \Omega 
		\end{cases}
	\end{equation}
	In other words, we numerically solve (\ref{min2ma}) with 
		\begin{equation*}
	{\bf f}_k= \sqrt[3]{(\Delta u_k(\xi_i) )^3+27(f(\xi_i)-det (D^2 u_k(\xi_i)))}
		\end{equation*}
	by using the iterative algorithm in \cite{LL21}.  \par 
	
	Terminate the iteration when
	$\|f-\text{det} (D^2 u_{k+1})\|_{l_\infty}>\|f-\text{det} (D^2 u_k)\|_{l_\infty}$.
\end{algorithm}

Next we explain an averaged iterative algorithm. 
Assume $\Omega$ is bounded and the closure of $\Omega$ is of uniformly positive reach as explained in the previous section. 
For any $f\in L^2(\Omega),$ the solution of the Poisson equation with zero boundary condition 
is in $H^2(\Omega)$ by Theorem~\ref{thm:regularity}.  
Furthermore,  the solution of the Poisson equation with boundary condition $g$ is in $H^2(\Omega)$ if $g\in H^{1/2}(\partial \Omega)$. Indeed, we consider a function $v\in H^2(\Omega)$ 
whose trace on $\partial \Omega$ is $g\in H^{1/2}(\partial \Omega)$. 
Define $w=u-v$ and we have
\begin{align*}
	\int_\Omega \nabla w \cdot \nabla \phi &=\int_\Omega \nabla u \cdot \nabla \phi -\int_\Omega \nabla v \cdot \nabla \phi \cr
	&=\int_\Omega -f \cdot  \phi -\int_\Omega \Delta v \cdot  \phi =\int_\Omega  (\Delta v-f )\cdot  \phi.
\end{align*}
for every $\phi\in H^1_0(\Omega)$. Then the solution $w$ satisfying the weak formulation of 
	\begin{equation*}
		 \Delta w=f-\Delta v~~\text{in} ~~\Omega, w=0~~\text{on} ~~\partial \Omega	\end{equation*}
is in $H^2(\Omega)$ (cf. \cite{GL20}). Therefore, $u=w+v$ is in $H^2(\Omega)$. 

Let $T$ be an operator which maps $H^2(\Omega) \to H^2(\Omega)$ in the following sense: 
for any $v\in H^2(\Omega)$, let $u=T(v)$ be the solution of the Poisson equation:
	\begin{equation*}
		\Delta u=\sqrt[3]{(\Delta v)^3+ 27(f- \det(D^2 v))} \hbox{ over } \Omega	\end{equation*}
and $u|_{\partial \Omega}=g$ with $g\in H^{1/2}(\partial \Omega)$.  In other words, 
the operator $T$ on $H^2(\Omega)$ is defined by
	\begin{equation*}
		T(u)=\Delta^{-1} [\sqrt[3]{(\Delta u )^3+27(f-\text{det} (D^2 u))}].	\end{equation*}
It is easy to see that $T$ is a nonlinear operator $T$ maps $H^2(\Omega)$ to $H^2(\Omega)$. 
Also, we can see that the exact
solution $u^*$ satisfying $\det(D^2 u^*)=f$ is a fixed point of $T$.  

Now we are ready to define an averaged iterative algorithm. 
In this way, we can find more accurate solutions than the one using Algorithm~\ref{alg1.1} only. 

\begin{algorithm}
	\caption{The Averaged Iterative Algorithm}
	\label{alg2}
	Start with an initial $u_0$, where $\Delta u_{0}=\sqrt[3]{27f}$ over $\Omega$ 
	and $u_0=g$ on $\partial \Omega$. \par
	
	We iteratively solve the Poisson equations  
	\begin{equation}
		\Delta u_{k+\frac{1}{2}}=\sqrt[3]{(\Delta u_k )^3+27(f-\text{det} (D^2 u_k))}, \quad   
	\end{equation}
	together with the boundary condition $u_{k+\frac{1}{2}}=g$ on $\partial \Omega$ 
	by using the minimization in (\ref{min2ma}) and then take
	\begin{equation}
		\label{average}
		u_{k+1}=\frac{1}{2}u_{k+\frac{1}{2}}+  \frac{1}{2}u_{k}.
	\end{equation}
	
	Stop the iteration if $\|f-\text{det} (D^2 u_{k+1})\|_{l_\infty}>\|f-\text{det} (D^2 u_k)\|_{l_\infty}$.
\end{algorithm}

Let us present some performance of these two algorithms to show that Algorithm~\ref{alg2} indeed very useful. Consider a testing function $u^{3ds1}$ as in Section \ref{sec:3}, 
the eigenvalues of the Hessian matrix $D^2 u^{3d1}$ are $1, 5, 15$. 
Although these three eigenvalues are not close  to any real positive number,  
we use various positive numbers $p$ for the right-hand side of the Poisson equation $\Delta u_0=p$ to solve $u_0$ as an initial 
solution and then apply Algorithm~\ref{alg1.1} and Algorithm~\ref{alg2}.  
In Table \ref{tableaaalgo}, the results from both Algorithms are shown after the same number
of iterations.   We can see that Algorithm~\ref{alg2} produces more accurate solution 
than Algorithm~\ref{alg1.1} from various initial values except for $p$ 
which is close to $21=\Delta u^{3ds1}$, i.e. $p\in [17.7, 26]$. 
Also,  the $\ell_2$ and $h_1$ errors from Algorithm~\ref{alg2} 
are better than the errors from Algorithm 1 for testing functions $u^{3ds3}$ and $u^{3ds8}$ as shown in  Figure \ref{figureaaalgo}.

\begin{table}[htpb]
	\centering
	\caption{Errors of numerical solutions $u^{3ds1}$ for  the Monge Amp\`ere  equation over $[0,1]^3$ with $D=9$, $r=1$ over the same tetrahedralization for various initial values $p$ by two algorithms}\label{tableaaalgo}
	\begin{tabular}{ c c  c c c } 
\toprule
\multirow{2}{1.4cm}{$\Delta u_0=p$}	&\multicolumn{2}{c}{Algorithm 1} &	\multicolumn{2}{c}{Algorithm 2} \\
\cmidrule(r{4pt}){2-3} \cmidrule(l){4-5}
			&$|e_s|_{l_2}$&$|e_s|_{h_1}$&$|e_s|_{l_2}$&$|e_s|_{h_1}$  \\
	\hline
		12.6 &  2.1291e-02 &  1.6315e-01 &  1.9230e-02 &  1.3670e-01\cr
		15.1 &  3.4135e-03 &  5.8902e-02 &  5.0004e-03 &  5.1503e-02\cr
		16.4 &  2.0124e-03 &  3.1978e-02 &  1.2081e-08 &  5.4356e-07\cr
		17.1 &  7.9130e-04 &  1.4517e-02 &  4.1401e-08 &  1.6448e-06\cr
		17.7 & 1.3980e-09 &  3.9929e-08 &  3.6103e-08 & 2.7446e-06\cr
		26.0 &  6.4074e-09  & 2.4003e-07&   4.8324e-07 &  1.8097e-05\cr
		26.5 &  2.0899e-04 &  6.3176e-03 &  5.0149e-07 &  1.8781e-05\cr
		27.0 &  5.6423e-04 &  1.7172e-02  & 3.9592e-04  & 1.2049e-02\cr
		27.5 &  8.9037e-04  & 2.5381e-02 &  7.0701e-04  & 2.0696e-02\cr
		\hline
	\end{tabular}
\end{table}

\begin{figure}[htpb]
	\includegraphics[width=0.85\linewidth, height=0.3\linewidth]{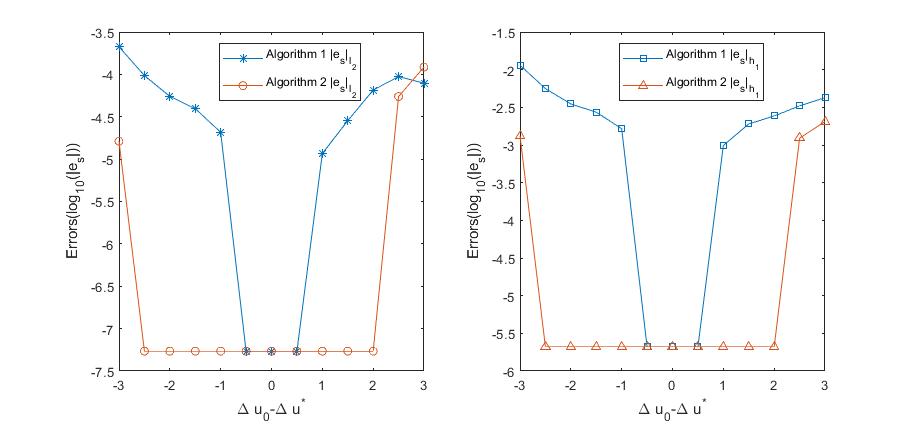}
	\centering
	\includegraphics[width=0.85\linewidth, height=0.3\linewidth]{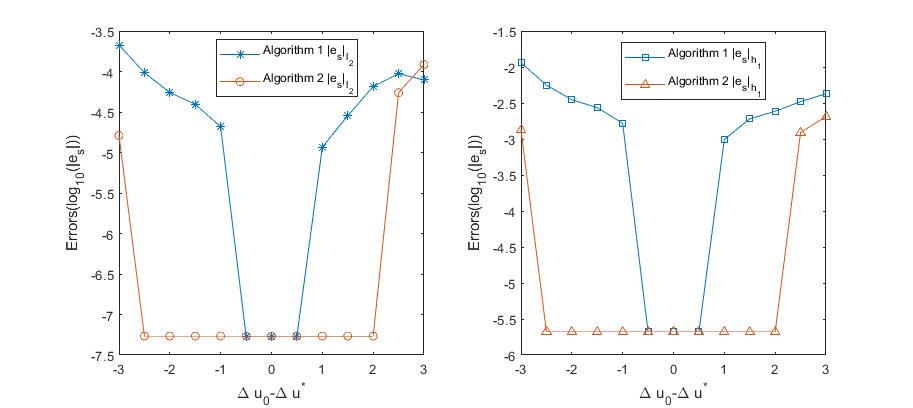}
	\caption{Errors $\log(|e_s|_{l_2}), \log(|e_s|_{h_1})$ for $u^{3ds3}$(Top) and $u^{3ds8}$(Bottom) }
\label{figureaaalgo}
\end{figure}

\subsection{Convergence Analysis}
\label{sec:2.3}
 According to \cite{A14}, it is known
that if $\text{det} (D^2 u^*)=f>0$ and $u^*$ is convex, then there exist constants $m, M>0$, independent of the mesh size $|\triangle|$ such that 
$$0<m \le \lambda_3\le \lambda_2\le \lambda_1 \le M, $$
where $\lambda_1,\lambda_2,\lambda_3$ are the eigenvalues of 
$ (D^2 u(x)), \forall x\in \Omega.$
The following result is also known (cf. \cite{A15}). For clarity, we provide a proof below. 
\begin{lemma}
	\label{eigen1}
	Suppose that the convex solution $u^*\in W^{2,\infty}$ satisfies $\text{det} (D^2 u^*)
=f>0$. There exists a $\delta>0$ such that  for any $u$ which is close 
	enough to the exact solution $u^*$ in the sense that $|u-u^*|_{2,\infty}\le \delta$, we have 
		\begin{equation*}
			\text{det} (D^2 u)\le \frac{1}{27}(\Delta u)^3< \frac{1}{a}(\Delta u)^3	\end{equation*}
	for any $a<27$.
\end{lemma}
\begin{proof}
Recall that the eigenvalues of a symmetric matrix are continuous functions of its entries, as 
roots of the characteristic equation (cf. Ostrowski (1960) Appendix K \cite{O60}). Thus, for a 
given $u^* \in W^{2,\infty}(\Omega),$ there exists $\delta>0$ such that for $u\in 
W^{2,\infty}(\Omega), |u-u^*|_{2,\infty}\le 
\delta$ implies $M\geq \lambda_1(D^2 u(x))\geq \lambda_2(D^2 u(x))\geq \lambda_3(D^2 u(x))>0.$
Now, we use the property that  $\text{det} (D^2 u)$ is the multiplication of all eigenvalues 
to have
	\begin{align*}
	&	\text{det} (D^2 u)=\lambda_1 \lambda_2 \lambda_3=\frac{1}{27}(3 \lambda_1 \lambda_2 \lambda_3+6\lambda_1 \lambda_2 \lambda_3+18\lambda_1 \lambda_2 \lambda_3)\\
		&\le \frac{1}{27}(\lambda_1^3+\lambda_2^3+\lambda_3^3+6\lambda_1 \lambda_2 \lambda_3 +3 \lambda_3 (\lambda_1^2+ \lambda_2^2)+3 \lambda_1 (\lambda_2^2+ \lambda_3^2)+3 \lambda_2 (\lambda_1^2+ \lambda_3^2))\\
		&=\frac{1}{27}(\lambda_1+\lambda_2+\lambda_3)^3=\frac{1}{27}(\Delta u)^3.
	\end{align*}
This completes all the proof.\qed
\end{proof}


We first consider the point-wise convergence of the sequence from Algorithm~\ref{alg1.1}.
\begin{theorem}
	\label{mainresult}
Fix a spline space $\mathcal{S}^r_D(\triangle)$ with $\triangle$ being a tetrahedralization of the domain
$\Omega$. 
Let  $u_k\in \mathcal{S}^r_D(\triangle), k\ge 1$ be the sequence from Algorithm 1. 
	Then, any average values of $f-\det (D^2 u_k), k\ge 1$ are nonnegative  in the following senses:
	\begin{equation}
		\label{mjlai03082022}
		\frac{1}{n+1}\sum_{k=0}^{n}(f (\bfx)-\det (D^2 u_k)(\bfx)) \ge 0, ~~ \bfx\in \Omega
	\end{equation} 
for all $n\ge 1$.  Furthermore,  suppose that there eixsts 
a bound $M>0$ such that $|u_k(\bfx)|\le M$ over $\Omega$ for all $k\ge 0$. Then 
\begin{equation}
		\label{mjlai03092022}
		\frac{1}{n+1}\sum_{k=0}^{n}(f(\bfx)-\det (D^2 u_k)(\bfx)) \to 0 
\end{equation}
when $n  \to \infty$ for all $\bfx\in \Omega$.
\end{theorem}
We remark that the condition that $|u_k(\bfx)| \le M$ above is a computational condition one can check during the iterative computation of Algorithm 1. 
Our numerical experiments show that for some testing functions $u$, this condition does satisfy while for other testing functions, the condition does not satisfy. 
See Figure~\ref{averageA} for these numerical phenomena.  
\begin{proof}
By \eqref{PDE2ma} and Lemma \ref{eigen1}, we get 
	\begin{align*}
		27\text{det} (D^2 u_{k+1}))\le(\Delta u_{k+1})^3&=(\Delta u_{k} )^3
+27(f-\text{det} (D^2 u_{k}))\cr
		&=(\Delta u_{k-1} )^3+27(f-\text{det} (D^2 u_{k-1}))+27(f-\text{det} (D^2 u_{k}))\\
		&=(\Delta u_{k-1} )^3+2\cdot 27 f-27\text{det} (D^2 u_{k-1}) - 27\text{det} (D^2 u_{k})\\
		&= \cdots \cr 
		&=(\Delta u_{0} )^3+27(k+1)f-27\sum_{j=0}^k\text{det} (D^2 u_{j})\cr
		&=27f+27(k+1)f-27\sum_{j=0}^k\text{det} (D^2 u_{j}).
	\end{align*}
	Hence, we have
	\begin{align*}
		0 \le 27f+27(k+1)f-27\sum_{j=0}^k\text{det} (D^2 u_{j})-27\text{det} (D^2 u_{k+1}))
		=27\sum_{j=0}^{k+1}(f-\text{det} (D^2 u_{j})).
	\end{align*}
	which leads to (\ref{mjlai03082022}). 
In addition,  we also have 
$$(\Delta u_{k+1})^3-(\Delta u_0)^3=27\sum_{j=0}^k (f-\text{det} (D^2 u_{j})). $$ 
By the assumption of this theorem, $u_{k+1}$ has a bound, i.e.  $\|u_{k+1}\|_{\infty, \Omega}\le M$.  
Then we can use the Markov inequality to have
\begin{equation}
\|\Delta u_{k+1}\|_{\infty,\Omega} \le \frac{C}{|\triangle|^2}\|u_{k+1}\|_{\infty,\Omega}\le  \frac{CM}{|\triangle|^2}<\infty 
\end{equation}
for a constant $C>0$ independent of $u_{k+1}$. 	It thus follows  
	$$\dfrac{27\sum_{j=0}^k (f-\text{det} (D^2 u_{j}))}{k+1}=\dfrac{(\Delta u_{k+1})^3-(\Delta u_0)^3}{k+1}\rightarrow 0.$$
Therefore, we finished a proof of Theorem~\ref{mainresult}.\qed 
\end{proof}

Furthermore, we denote $w(u,f):=\sqrt[3]{(\Delta u )^3+27(f-\text{det} (D^2 u))}$. 
We have
\begin{align*}
	\|\Delta u_{k+1}-\Delta u \|_{L^2(\Omega)}&=\| \dfrac{(\Delta u_{k} )^3+27(f-\text{det} 
		(D^2 u_{k}))-(\Delta u)^3}{(w(u_k,f))^2+w(u_k,f)w(u,f)+(w(u,f))^2}\|_{L^2(\Omega)}\cr
	&=\| \dfrac{(\Delta u_{k} )^3-(\Delta u)^3+27(\text{det} (D^2 u)-\text{det} (D^2 u_{k}))}{(w(u_k,f))^2+w(u_k,f)w(u,f)+(w(u,f))^2}\|_{L^2(\Omega)}
\end{align*}
By simple calculations, we get
\begin{equation*}
	(\Delta u_k)^3-(\Delta u)^3=(\Delta u_k-\Delta u)((\Delta u_k)^2+\Delta u_k \cdot \Delta u +(\Delta u)^2)
\end{equation*}	
and by Lemmas 2.1, 2.2 and 2.3 in \cite{A14}
\begin{equation*}
	\text{det} (D^2 u)-\text{det} (D^2 u_{k})=\text{cof}((1-t)D^2 u_k+tD^2u):(D^2u_k-D^2u)
\end{equation*}
for some $t\in [0,1].$ By simple calculation and Lemma \ref{mjlai03222021}, we have
\begin{align*}
	&\| \dfrac{(\Delta u_{k} )^3-(\Delta u)^3}{(w(u_k,f))^2+w(u_k,f)w(u,f)
		+(w(u,f))^2}\|_{L^2(\Omega)}\cr
	&\le \| \dfrac{(\Delta u_k-\Delta u)((\Delta u_k)^2+\Delta u_k \cdot \Delta u +(
		\Delta u)^2)}{(w(u_k,f))^2+w(u_k,f)w(u,f)+(w(u,f))^2}\|_{L^2(\Omega)}\cr
	&\le \| \dfrac{(\Delta u_k)^2+\Delta u_k \cdot \Delta u 
		+(\Delta u)^2}{(w(u_k,f))^2+w(u_k,f)w(u,f)+(w(u,f))^2}\|_{L^\infty(\Omega)} 
	\|\Delta u_k-\Delta u\|_{L^2(\Omega)}.
\end{align*}
Let $M=(M_{ij}), N=(N_{ij})$ be matrix fields and $c$ be a real number. Then we get 
\begin{align}
	\label{inmatrix}
	\dfrac{M:N}{c}&= \frac{1}{c}\sum_{i,j=1}^3 M_{ij} N_{ij}= \sum_{i,j=1}^3\frac{M_{ij}}{c}  
	N_{ij}\cr
	&\le \|\frac{M}{c}\|_\infty\sum_{i,j=1}^3 | N_{ij}|\le \|\frac{M}{c}\|_\infty \Big{(}\sum_{i,j=1}^3  
	N_{ij}^2\Big{)}^{1/2} \cdot 3, 
\end{align}
where $\|\frac{M}{c}\|_\infty=\max_{1\le i\le 3}  \sum_{j=1}^3|\frac{M_{ij}}{c}|. $ 
By \eqref{inmatrix} with , we have
\begin{align*}
	&\| \dfrac{27(\text{det} (D^2 u)-\text{det} (D^2 
		u_{k}))}{(w(u_k,f))^2+w(u_k,f)w(u,f)+(w(u,f))^2}\|_{L^2(\Omega)}\cr
	&\le 27 \| \dfrac{(\text{cof}((1-t)D^2 
		u_k+tD^2u):(D^2u_k-D^2u))}{(w(u_k,f))^2+w(u_k,f)w(u,f)+(w(u,f))^2}\|_{L^2(\Omega)}\cr
	&=27\Big{[}\int \Big{(}\dfrac{(\text{cof}((1-t)D^2 
		u_k+tD^2u):(D^2u_k-D^2u))}{(w(u_k,f))^2+w(u_k,f)w(u,f)+(w(u,f))^2}\Big{)}^2\Big{]}^{\frac{1}{2}}\cr
	&= 81\| \dfrac{\text{cof}((1-t)D^2 
		u_k+tD^2u)}{(w(u_k,f))^2+w(u_k,f)w(u,f)+(w(u,f))^2}\|_\infty |u_k-u|_{H^2(\Omega)} \cr
	&\le 81\| \dfrac{(1-t)D^2 u_k+tD^2u}{(w(u_k,f))^2+w(u_k,f)w(u,f)+(w(u,f))^2}
	\|^2_\infty |u_k-u|_{H^2(\Omega)}\cr
	&\le 81\| \dfrac{(1-t)D^2 u_k+tD^2u}{(w(u_k,f))^2+w(u_k,f)w(u,f)+(w(u,f))^2}
	\|^2_\infty \|u_k-u\|_{H^2(\Omega)}
\end{align*}
for some $t\in [0,1]$. 	By these two equations, we can have 
\begin{align*}
	\|\Delta u_{k+1}-\Delta u \|_{L^2(\Omega)}&=\| \dfrac{(\Delta u_{k} )^3-(\Delta u)^3+
		27(\text{det} (D^2 u)-\text{det} (D^2 u_{k}))}{(w(u_k,f))^2+w(u_k,f)w(u,f)+(w(u,f))^2}
	\|_{L^2(\Omega)}\cr
	&\le \| \dfrac{(\Delta u_k-\Delta u)((\Delta u_k)^2+\Delta u_k \cdot \Delta u +
		(\Delta u)^2)}{(w(u_k,f))^2+w(u_k,f)w(u,f)+(w(u,f))^2}\|_{L^2(\Omega)}\cr
	&+\| \dfrac{27(\text{det} (D^2 u)-\text{det} (D^2 u_{k}))}{(w(u_k,f))^2+w(u_k,f)w(u,f)+
		(w(u,f))^2}\|_{L^2(\Omega)}\cr
	&\le \| \dfrac{(\Delta u_k)^2+\Delta u_k \cdot \Delta u +(\Delta 
		u)^2}{(w(u_k,f))^2+w(u_k,f)w(u,f)+(w(u,f))^2}\|_{L^\infty(\Omega)} \|\Delta u_k-\Delta u\|_{L^2(
		\Omega)}\cr
	&+81 \| \dfrac{(1-t)D^2 u_k+tD^2u}{(w(u_k,f))^2+w(u_k,f)w(u,f)+(w(u,f))^2}\|^2_\infty \|u_k-u\|_{H^2(\Omega)}
\end{align*}
for some $t\in [0,1]$.
	Now, we need  the following lemma from  \cite{LL21} 
to prove one of the main convergence results in this paper. 
\begin{lemma}
	\label{mjlai03222021}
	Suppose that $\Omega$ is bounded and has  uniformly positive reach $r_\Omega>0.$ 
	Then there exist two positive constants $A$ and $B$ such that 
	\begin{equation}
		\label{equivalent4}
		A \|u \|_{H^2(\Omega)} \le \|\Delta u \|_{L^2(\Omega)} \le B \| u\|_{H^2(\Omega)}, \quad 
		\forall u\in H^2(\Omega)\cap H^1_0(\Omega).
	\end{equation}
\end{lemma}
By Lemma \ref{mjlai03222021}, we have
\begin{align*}
	\|\Delta u_{k+1}-\Delta u \|_{L^2(\Omega)}	&\le \| \dfrac{(\Delta u_k)^2+\Delta u_k \cdot \Delta u 
		+(\Delta u)^2}{(w(u_k,f))^2+w(u_k,f)w(u,f)+(w(u,f))^2}\|_{L^\infty(\Omega)} 
\|\Delta u_k-\Delta u\|_{L^2(
	\Omega)}\cr
	&+81\| \dfrac{(1-t)D^2 u_k+tD^2u}{(w(u_k,f))^2+w(u_k,f)w(u,f)+(w(u,f))^2}
	\|^2_\infty \frac{1}{A} \|\Delta u_k-\Delta u\|_{L^2(
		\Omega)}
\end{align*}
and therefore
\begin{align*}
	\|\Delta u_{k+1}-\Delta u \|_{L^2(\Omega)}	&\le \rho_k  \|\Delta u_k-\Delta u\|_{L^2(
		\Omega)} , 
\end{align*}
where 
\begin{eqnarray}
\label{rhok}
	\rho_k:&=& \| \dfrac{(\Delta u_k)^2+\Delta u_k \cdot \Delta u +(\Delta u)^2}{(w(u_k,f))^2+w(u_k,f)w(u,f)+(w(u,f))^2}\|_{L^\infty(\Omega)}+ \cr
	&& \frac{81}{A} \| \dfrac{(1-t)D^2 u_k+tD^2u}{(w(u_k,f))^2+w(u_k,f)w(u,f)+(w(u,f))^2}\|^2_\infty.
\end{eqnarray}
We are now ready to conclude the following result 
\begin{theorem}\label{mainthm2}
	Suppose that $\Omega$ is bounded and has  uniformly positive reach. 
	If $\rho_k\le \gamma <1$ for all $k\ge 1$, 
	then the sequence $\{u_k\}$ from Algorithm~\ref{alg1.1} converges.
\end{theorem}

Note that our numerical experiments show that for some testing function $u$, we have indeed 
$\rho_k<1$ while there is other testing function $u$ which gives $\rho_k>1$. 
See Figures~\ref{averageA} and \ref{averageApart}. Also, it is hard to estimate $\rho_k$ from the formula (\ref{rhok}). 
\begin{figure}[htpb]
		\includegraphics[height=0.4\linewidth, width=1\linewidth]{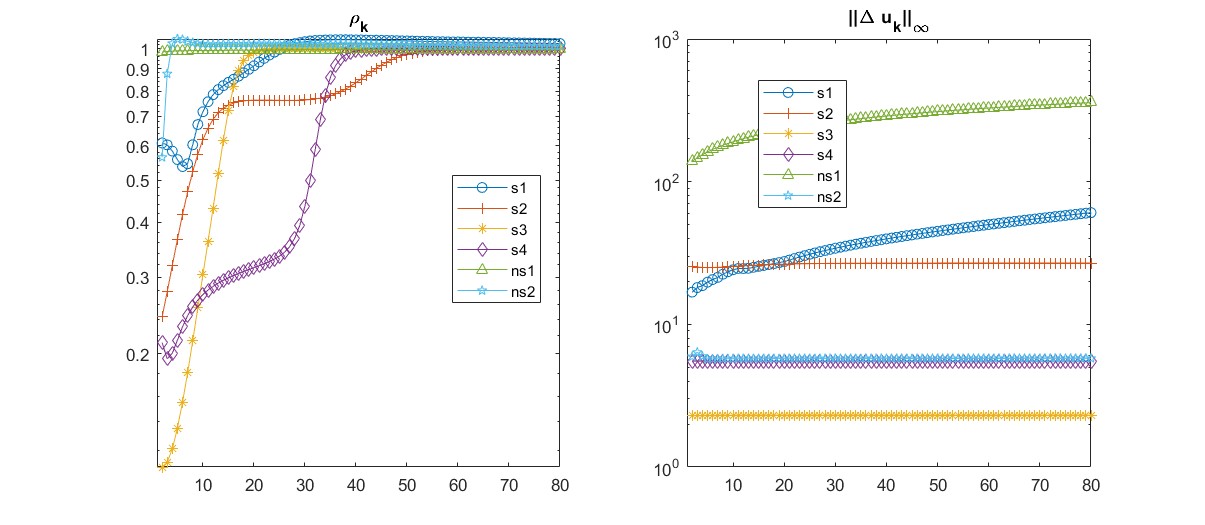}
	\caption{$\rho_k$ and $\|\Delta u_k\|_{\infty}$  for 80 iterations for smooth solutions $s_1,s_2, s_3$ and 
		non-smooth solutions $ns_1, ns_2$ }\label{averageA}
\end{figure}
\begin{figure}[htpb]
		\includegraphics[height=0.4\linewidth, width=1\linewidth]{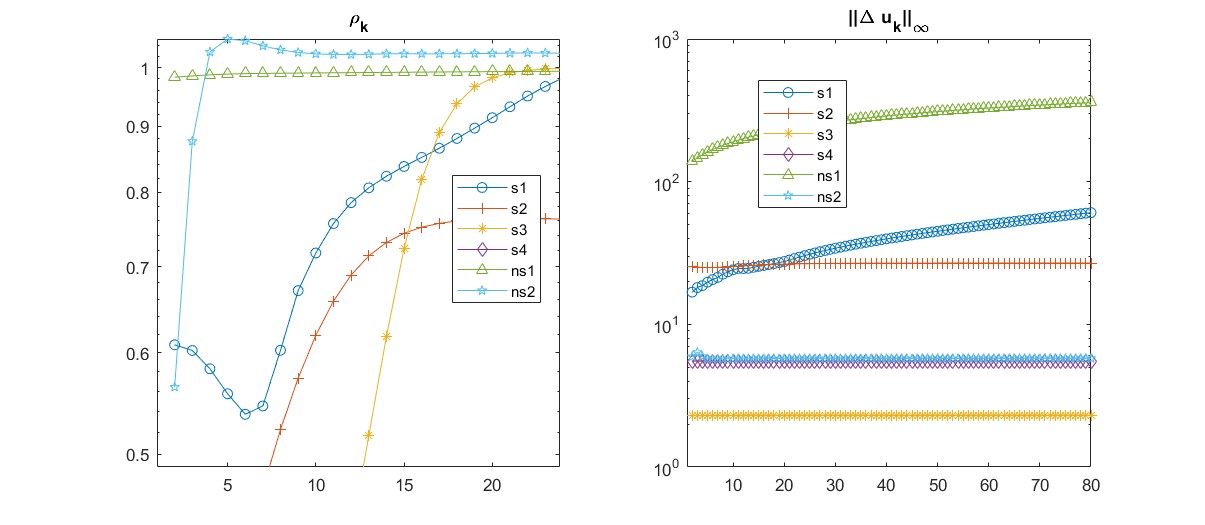}
	\caption{An enlarged graphs in Figure \ref{averageA} 
	}\label{averageApart}
\end{figure}

In Figures \ref{averageA}  and \ref{averageApart}, we plot $\rho_k$ 
corresponding to the numerical solution for smooth solutions $s_1,s_2, s_3,s_4$ 
and non-smooth solutions $ns_1, ns_2$. They are defined as follows.

\begin{itemize}
\item $s_1$: polynomial function $(x^2+5y^2+15z^2)/2$;
\item $s_2$: exponential function $\exp((x^2+y^2+z^2)/2)$;
\item $s_3$: radical function $-\sqrt{6-(x^2+y^2+z^2)}$;
\item $s_4$: $(x^2+y^2+z^2)/2-\sin(x)-\sin(y)-\sin(z)$;
\item $ns_1$: $-\sqrt{3-(x^2+y^2+z^2)}$ where $f$ is $\infty$ at $(1,1,1)$; 
\item $ns_2$: $\dfrac{(x^2+y^2+z^2)^{3/4}}{3}$ where $f$ is $\infty$ at $(0,0,0)$.
\end{itemize}
The graphs in these  figure above  show that $\rho_k<1$  and $\|\Delta u_k\|_{\infty}$ are bounded for smooth testing solutions.
However, $\rho_k$ may be bigger than $1$ and $\|\Delta u_k\|_{\infty}$ may increase which maybe unbounded for nonsmooth testing functions.

When $\rho_k>1$, the above analysis will not be useful to see convergence of the sequence $\{u_k\}$. 
The remaining case is  $\rho_k \le 1$.  In this case, we need Algorithm~\ref{alg2}.
That is,  we now study the convergence of our Algorithm~\ref{alg2}.  Letting 
$u^k, k\ge 1$ be the sequence from Algorithm~\ref{alg2}, it is easy to see that 
\begin{equation}
	\label{bound}
	u_{k+1}- u^* = \frac{1}{2}(u_{k}- u^*) +\frac{1}{2}(T(u_{k})- T(u^*))
\end{equation}
for all $k\ge 1$ since $u^*$ is a fixed point of $T$. 
Since $\rho_k \le 1$, we have $\|T(u_{k})- T(u^*)\|
\le \|u_{k}- u^*\|$ and hence, $\|u_{k+1}- u^* \|\le \|u_{k}- u^* \|$ for all $k\ge 1$. 
It follows that $u_k, k\ge 1$ are bounded in a $H^2(\Omega)$ norm.  
We now show the  averaged iterative algorithm converges. 
\begin{theorem}
	Suppose that $\Omega$ is a bounded domain which has a uniformly positive reach. Suppose that
	$g\in H^{1/2}(\partial\Omega)$.  Suppose that $\rho_k\le 1$.  
	Then  Averaged Iterative Algorithm~\ref{alg2} converges.
\end{theorem}
\begin{proof}
	Let $S=H^2(\Omega)$. 
	By the assumptions, the operator $T$ defined above from 
	$S$ to $S$ is a continuous and nonexpansive operator.   
	We first recall the following equality: For any $x,y,z\in S$ and a real number 
	$\lambda \in [0, 1],$ we have the following identity
	\begin{align*}
		\lambda \|x-z\|^2+(1-\lambda) \|y-z\|^2-\lambda (1-\lambda)\|x-z\|^2=\|\lambda x+(1-
		\lambda)y-z\|^2.
	\end{align*}
	The proof is left to the interested reader. Let $\lambda=1/2$ and $x=u^k, y=T(u^k),$ 
	and $z=u^*$ which is a fixed point or the solution. Then we have
	\begin{align*}
		\|u_{k+1}-u^*\|^2&=\frac{1}{2}\|u_{k}-u^*\|^2+\frac{1}{2}\|T(u_{k})-u^*\|^2
		-\frac{1}{4}\|u_{k}-T(u_{k})\|^2\cr
		&=\frac{1}{2}\|u_{k}-u^*\|^2+\frac{1}{2}\|T(u_{k})-T(u^*)\|^2-\frac{1}{4}\|u_{k}-T(u_{k})
		\|^2\cr
		&\le \|u_{k}-u^*\|^2-\frac{1}{4}\|u_{k}-T(u_{k})\|^2.
	\end{align*}
	It follows that
		\begin{equation*}
			\sum_{k=1}^N \frac{1}{4}\|u_{k}-T(u^{k})\|^2+\|u_{N+1}-u^*\|^2\le \|u^{0}-u^*\|^2	\end{equation*}
	for any integer $N>1$.  
	That is, $\|u_{k}-T(u_{k})\|\to 0$ when $k\to \infty$. 
	
We now claim that the sequence $u_k, k\ge 1$ converges. 
Note that due to the nonexpansiveness, $\|u_k\|, k\ge 1$ are bounded as explained above.  
	Let $\hat{u}$ be the limit of a subsequence of 
	$u^k, k\geq 1.$ Then we have $\hat{u}=T(\hat{u})$ by the continuity of the operator $T.$ 
	So $\hat{u}$ is a fixed point of $T$. By the definition of $T$, we have
		\begin{equation*}
	\Delta \hat{u} = \sqrt[3]{(\Delta \hat{u})^3+ 27(f- \det (D^2 \hat{u}))}
		\end{equation*}
	or $(\Delta \hat{u})^3 =(\Delta \hat{u})^3+ 27(f- \det (D^2 \hat{u}))$. It follows 
	that $f= \det (D^2 \hat{u})$.  Since the \MAE has a unique solution, we have $\hat{u}=u^*$.
	If there exists another  $\tilde{u}$ which is the limit of another subsequence of 
	$u^k, k\ge 1$, we also have $\tilde{u}=T(\tilde{u})$.  Then $\tilde{u}=u^*$.
	Hence, the sequence $\{u_k, k\ge 1\}$ from Algorithm~\ref{alg2} converges. \qed
\end{proof}

\section{Numerical Results for 3D Monge-Amp\`ere Equations}
\label{sec:3}
In this section, we present numerical results from various computational experiments. We will 
first test several smooth and nonsmooth solutions over convex domains, such as $[0, 1]^3$.   Next, we show the numerical results over non-convex 
domains such as $C, L, S$-shaped domains. 
For all the experiments, we use 8 processors on a parallel computer, which has AMD Ryzen 7 4800H with Radeon Graphics 2.90 GHz.  
All the cases, the errors are computed based on $NI=351 \times 351  \times 351 $ equally spaced points 
$\{(\eta_i)\}_{i=1}^{NI}$ fell inside the  domain of computation.
The errors will be calculated according to the norms
\begin{align*}
	\begin{cases}
		|u|_{l_2}&=\sqrt{\frac{\sum_{i=1}^{NI} (u(i))^2}{NI} }\cr
		|u|_{h_1}&=\sqrt{\frac{\sum_{i=1}^{NI} (u(i))^2+(u_x(i))^2+(u_y(i))^2+(u_z(i))^2}{NI} }\cr
		|u|_{l_\infty}&= \max |u(i)|,
	\end{cases}
\end{align*}
where $u(i):=u(\eta_i), u_x(i):=u_x (\eta_i)$, $u_y(i):=u_y(\eta_i)$ and 
$u_z(i):=u_z(\eta_i)$ for given functions $u, 
u_x, u_y, u_z.$ Tables in this section are the numerical results of  $ |e_s|_{l_2}$ 
and $|e_s|_{h_1}$, where $e_s:=u-u_s$.

\subsection{Smooth Testing Functions}
\label{sec:3.1}
\begin{example}[Polynomial Examples]
	In \cite{CGG18}, the researchers experimented the following two smooth exact solutions:  
	\begin{itemize}
		\item $f^{3d1}=75$ such that an exact solution is $u^{3ds1}=\frac{1}{2}(x^2+5y^2+15 z^2)$
		\item $f^{3d2}=1000$ such that an exact solution is $u^{3ds2}=\frac{1}{2}(x^2+10y^2+100 z^2)$
	\end{itemize}
Numerical results of their least squares/relaxation method (called LR method in this paper) are shown in 
Table~\ref{table3dcompare0}. Together we present numerical results based on  our spline collocation method  
by Algorithms 2.  
\begin{table}[htpb]
		\centering
\caption{Errors of numerical solutions $u^{3ds1}, u^{3ds2}$ for Monge Amp\`ere  equation over $[0,1]^3$ 
for LL methods with $D=5$, $r=1$ and LR method in \cite{CGG18}}\label{table3dcompare0}
		\begin{tabular}{ c c  c c c } 
	\toprule
			\multicolumn{5}{c}{LR method}  \\
			\hline
				\multirow{2}{0.5cm}{$h$}&	\multicolumn{2}{c}{$u^{3ds1}$}&	\multicolumn{2}{c}{$u^{3ds2}$}  \\
		 \cmidrule(lr){2-3} \cmidrule(r){4-5}
		&$|e_s|_{l_2}$&$|e_s|_{h^1}$&$|e_s|_{l_2}$&$|e_s|_{h^1}$  \\
		\bottomrule
			0.2&7.19e-02&1.58e-00&2.74e-02&5.16e-01\cr
			0.1&1.80e-02&7.91e-01&7.52e-03&2.81e-01\cr
			0.0625&7.06e-03&4.95e-01&3.06e-03&1.83e-01\cr
			0.04&2.89e-03&3.16e-01&1.26e-03&1.20e-01\cr
		\toprule
			\multicolumn{5}{c}{LL method}  \\
			\hline
				\multirow{2}{0.5cm}{$h$}&	\multicolumn{2}{c}{$u^{3ds1}$}&	\multicolumn{2}{c}{$u^{3ds2}$}  \\
		\cmidrule(lr){2-3} \cmidrule(r){4-5}
		&$|e_s|_{l_2}$&$|e_s|_{h^1}$ &$|e_s|_{l_2}$&$|e_s|_{h^1}$ \\
			\bottomrule
			0.25&2.68e-07  & 1.01e-05&2.48e-04  & 4.63e-03 \\
			\hline
		\end{tabular}
	\end{table}
Table \ref{table3dcompare0} shows our spline collocation method (called LL method) 
produces more accurate solutions than those presented in \cite{CGG18}.  
The eigenvalues of the Hessian matrix are $1, 5, 15$ and 
therefore $\text{det}(D^2 u^{3ds1})=\lambda_1 \lambda_2 \lambda_3=75$ and $\Delta u^{3ds1}=\lambda_1+\lambda_2+\lambda_3=21.$ In  Algorithm~\ref{alg2}, 
we choose an initial value $\Delta u_{0}=14.55$ to approximate the exact solution $u^{3ds1}$. This choice of initial value leads to converging iterations since 14.55 is close to $\Delta u^{3ds1}=\sqrt[3]{27f}=\sqrt[3]{27\cdot 75}=12.65$.  
Similarly, we choose our initial value $u_0$ for $ u^{3ds2}$ satisfying $\Delta u_{0}=106.2$ 
which makes the iterations from Algorithm~\ref{alg2} converge. 
By choosing good initial value $u_0$ we achieve the numerical results shown in Table \ref{table3dcompare0}.
\end{example}

We also test other smooth solutions which were experimented in the literature, e.g., \cite{A13P},
\cite{A15}, \cite{CGG18}, \cite{VNNP20}, and etc..
\begin{example}[Smooth Exponential Functions]
	Consider a smooth exponential exact solution $u^{3ds3}=e^{\frac{(x^2+y^2+z^2)}{2}}$ associated with 
	$f^{3d3}=(1+x^2+y^2+z^2)e^{\frac{3(x^2+y^2+z^2)}{2}}$.   
	We compare our methods with the least squares/relaxation method(LR method) in \cite{CGG18}.
	Table \ref{table3dcompare1} shows comparison results including $l_2, h_1$ norm 
	of these two methods for each mesh size $h$. 
	
	\begin{table}[htpb]
			\caption{Errors of numerical solutions $u^{3ds3}$ for Monge Amp\`ere  equation over $[0,1]^3$ for LL methods with $D=5$, $r=1$ and LR method in \cite{CGG18}}\label{table3dcompare1}
				\centering
		\begin{tabular}{c c}	
			\begin{tabular}{ c c  c c c } 
			\toprule
				\multicolumn{5}{c}{LR method}  \\
				\hline
				\multicolumn{1}{c}{$h$}&$|e_s|_{l_2}$&rate&$|e_s|_{h^1}$&rate  \\
				\hline
				0.2&7.19e-02&-&1.58e-00&-\cr
				0.1&1.80e-02&1.99&7.91e-01&0.99\cr
				0.0625&7.06e-03&1.99&4.95e-01&1.00\cr
				0.04&2.89e-03&1.99&3.16e-01&0.99\cr
				\hline
			\end{tabular}
			&
			\begin{tabular}{ c c  c c c } 
				\toprule
				\multicolumn{5}{c}{LL method}  \\
				\hline
				\multicolumn{1}{c}{$h$} &$|e_s|_{l_2}$&rate&$|e_s|_{h^1}$&rate  \\
			\hline
				1& 1.17e-03  &  - & 1.24e-02 &  -  \cr  
				0.5 & 4.36e-05  &4.74 & 7.82e-04 &3.99  \cr  
				0.25 & 1.42e-06  &4.94 & 2.72e-05 &4.84  \cr  
				0.125 & 1.10e-07  &3.69 & 1.36e-06 &4.32  \cr  
			\hline
			\end{tabular}	
		\end{tabular}
	\end{table}
	
	We can see that  better convergence results using LL methods with $D=5, r=1.$ 
	In Figure \ref{grapherr1}, we show plots of the $|e_s|_{l_2}, |e_s|_{h^1}$ with respect to the mesh size $h$. We can see that the rate of convergences is about $\mathcal{O}(h^{4.82}).$ 
\begin{figure}[htpb]
	\centering
	\begin{tabular}{cc}
		\includegraphics[width=.4\linewidth]{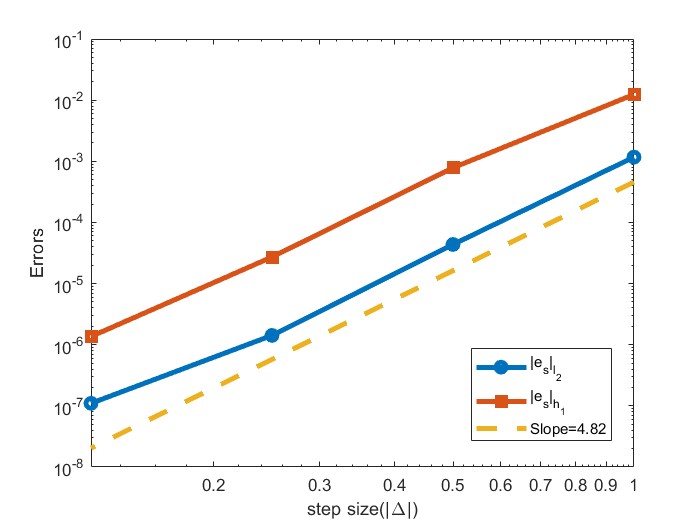}
		&
		\includegraphics[width=.4\linewidth]{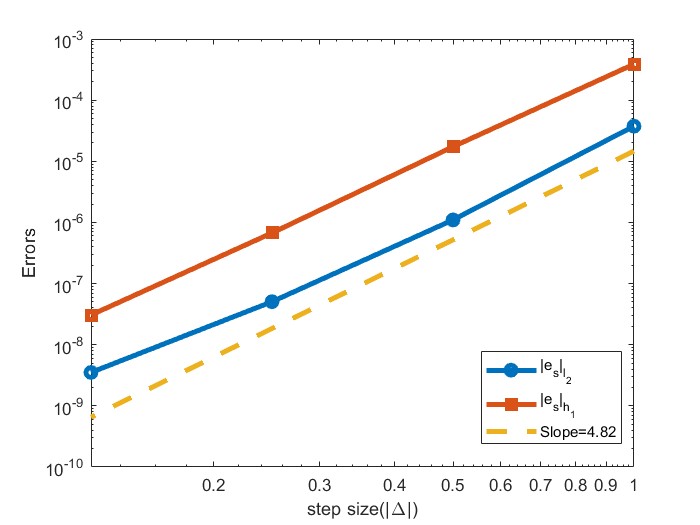}
		\cr
	\end{tabular}
	\caption{Convergence rates of $l_2, h_1$ errors for solutions $u^{3ds3}$(Left) and $u^{3ds5}$(Right) with respect to $|\triangle|$ 
	based on the LL method}\label{grapherr1}
\end{figure}
\end{example}

\begin{example}
	In \cite{A13P} and \cite{A15}, Awanou introduced the pseudo transient continuation, time marching methods and the spline element 
	methods for Monge-Amp\'ere equations. He presented several 2D and 3D numerical examples by his methods. 
	For testing function $u^{3ds4}=e^{\frac{(x^2+y^2+z^2)}{3}}$, it seems that 
	the numerical results using the spline element method (SE method) in \cite{A15} is the best.  We use our spline
	collocation method(LL method) and compare $L^2, H^1, H^2$ errors of our method and the SE method.  
	Table \ref{table3dcomparea1} and \ref{table3dcomparea2} show that we can 
	get better accuracy when $h=1$ and $h=1/2$ for each degree $D=4, 5, 6$.
	
	\begin{table}[htpb]
		\centering
			\caption{Errors of numerical solutions $u^{3ds4}$ for Monge Amp\`ere  equation over $[0,1]^3$ for LL methods with $D=3,4,5,6$, $r=1, h=1$ and SE method in \cite{A15}}\label{table3dcomparea1}
		\begin{tabular}{ c  c c c  c c c} 
		\toprule
		\multirow{2}{.4cm}{$D$}&\multicolumn{3}{c}{SE method} &	\multicolumn{3}{c}{LL method}   \\
		 \cmidrule(lr){2-4} \cmidrule(r){5-7}
		&$L^2$ norm&$H^1$ norm &$H^2$ norm &$L^2$ norm&$H^1$ norm &$H^2$ norm\\
	\midrule
			3  & 1.2338e-02 & 7.6984e-02  & 4.4411e-01 	 & 1.6916e-02 & 1.0879e-01  & 3.8250e-01 \cr  
			4  & 1.6289e-03 & 1.4719e-02  & 1.3983e-01 & 6.4696e-04 & 6.1874e-03  & 3.7146e-02 \cr   
			5  & 1.5333e-03 & 8.7312e-03  & 6.0412e-02 & 1.7440e-04 & 2.2203e-03  & 1.7392e-02 \cr   
			6 & 1.2324e-04 & 9.7171e-04  & 1.0584e-02  & 4.6740e-05 & 6.2257e-04  & 3.5432e-03 \cr  
			\bottomrule
		\end{tabular}
	\end{table}
	\begin{table}[htpb]
		\centering
			\caption{Errors of numerical solutions $u^{3ds4}$ for Monge Amp\`ere  equation over $[0,1]^3$ for LL methods with $D=3,4,5,6$, $r=1, h=1/2$ and SE method in \cite{A15}}\label{table3dcomparea2}
		\begin{tabular}{ c  c c c  c c c} 
		\toprule
		\multirow{2}{.4cm}{$D$}&\multicolumn{3}{c}{SE method} &	\multicolumn{3}{c}{LL method}   \\
		\cmidrule(lr){2-4} \cmidrule(r){5-7}
		&$L^2$ norm&$H^1$ norm &$H^2$ norm &$L^2$ norm&$H^1$ norm &$H^2$ norm\\
		\midrule
			3  & 3.1739e-03 & 2.3005e-02  & 2.4496e-01 	& 2.4294e-03 & 1.5806e-02  & 1.0139e-01 \cr  
			4  & 3.2786e-04 & 3.5626e-03  & 5.2079e-02 & 9.5591e-05 & 1.1644e-03  & 9.8077e-03 \cr  
			5  & 2.4027e-05 & 3.9210e-04  & 8.8868e-03 & 5.8750e-06 & 1.2214e-04  & 1.4292e-03 \cr  
			6 & 1.3821e-06 & 2.2369e-05  & 6.0918e-04 & 6.0635e-07 & 1.4198e-05  & 1.6487e-04 \cr    
				\bottomrule
		\end{tabular}
	\end{table}
\end{example}

\subsection{Non-smooth Testing Functions}
\label{sec:3.2}
\begin{example}
	In \cite{CGG18}, the researchers considered the following problem which do not have exact solution with the 
	$H^2(\Omega)-$ regularity or may have no solution at all. For $R\geq \sqrt{3},$ let 
	$u=-\sqrt{R^2-(x^2+y^2+z^2)}$ be a testing function.  
	When $R>\sqrt{3},$ this function belongs to $C^\infty (\bar{\Omega})$, while $u\in C^0(\bar{\Omega})\cap 
	W^{1,s}(\Omega),$ with $1\le s<2,$ if $R=\sqrt{3}.$ More precisely, let us consider the following two solutions  
	$$u^{3ds5}=-\sqrt{6-(x^2+y^2+z^2)} ~~\text{with}~~ f^{3d5}=6(6-(x^2+y^2+z^2))^{-\frac{5}{2}}$$ and $$u^{3ds6}=-\sqrt{3-(x^2+y^2+z^2)} ~~\text{with}~~ f^{3d6}=3(3-(x^2+y^2+z^2))^{-\frac{5}{2}}.$$
	The numerical results from the least squares/relaxation method in \cite{CGG18} (called LR method) 
	are shown in Table~\ref{table3dcompare2}. In Figure \ref{grapherr1}, we can see that the rate of convergences of $u^{3ds5}$ are about $\mathcal{O}(h^{4.82}).$ In addition, we show our spline collocation method (called LL method) in the
	same table for comparison.  
	
	\begin{table}[htpb]
			\caption{Errors of numerical approximation of the solution $u^{3ds5}$ for Monge Amp\`ere  equation over $[0,1]^3$ by the and LR method and by the LL method with $D=5$, $r=1$}\label{table3dcompare2}
			\centering
		\begin{tabular}{c c}
			\begin{tabular}{ c c  c c c } 
					\toprule
				\multicolumn{5}{c}{LR method}  \\
				\hline
				\multicolumn{1}{c}{$|\triangle|$}&$|e_s|_{l_2}$&rate&$|e_s|_{h^1}$&rate  \\
			\midrule
				0.2&4.96e-03&-&8.60e-02&-\cr
				0.1&1.28e-03&1.95&4.41e-02&0.96\cr
				0.0625&5.09e-04&1.96&2.78e-02&0.97\cr
				0.04&2.10e-04&1.97&1.79e-02&0.98\cr
			\bottomrule
			\end{tabular}
			&
			\begin{tabular}{ c c  c c c } 
			\toprule
			\multicolumn{5}{c}{LL method}  \\
				\hline
				\multicolumn{1}{c}{$|\triangle|$} &$|e_s|_{l_2}$&rate&$|e_s|_{h^1}$&rate  \\
				\midrule
				1 & 3.75e-05  &  - & 3.88e-04 &  -  \cr  
				0.5 & 1.10e-06  &5.09 & 1.73e-05 &4.49  \cr  
				0.25& 5.05e-08  &4.45 & 6.73e-07 &4.69  \cr  
				0.125 & 3.52e-09  &3.84 & 3.06e-08 &4.46  \cr  
				\bottomrule
			\end{tabular}	
		\end{tabular}
	\end{table}

	\begin{table}[htpb]
			\caption{Errors of numerical approximation of the solution $u^{3ds6}$ for Monge Amp\`ere  equation over $[0,1]^3$ by the and LR method and by the LL method with $D=5$, $r=1$}\label{table3dcompare3}
					\centering
		\begin{tabular}{c c}
			\begin{tabular}{ c c  c c c } 
			\toprule
			\multicolumn{5}{c}{LR method}  \\
				\hline
				\multicolumn{1}{c}{$|\triangle|$}&$|e_s|_{l_2}$&rate&$|e_s|_{h^1}$&rate  \\
			\midrule
				0.2&1.15e-02&-&6.60e-01&-\cr
				0.1&3.06e-03&1.91&6.31e-01&-\cr
				0.0625&1.24e-03&1.92&6.25e-01&-\cr
				0.04&5.17e-04&1.96&6.22e-01&-\cr
			\bottomrule
			\end{tabular}
			&
			\begin{tabular}{ c c  c c c } 
			\toprule
			\multicolumn{5}{c}{LL method}  \\
				\hline
				\multicolumn{1}{c}{$|\triangle|$} &$|e_s|_{l_2}$&rate&$|e_s|_{h^1}$&rate  \\
			\midrule
				1 & 8.07e-02  &  - & 7.02e-01 &  -  \cr  
			0.5 & 7.06e-03  &3.52 & 1.63e-01 &2.10 \cr  
			0.25 & 4.78e-04  &3.88 & 2.21e-02 &2.89  \cr  
			0.125 & 3.85e-04  &0.31 & 2.54e-02 &-0.20  \cr  
			0.0625&3.57e-04&0.11 &  1.98e-02&0.35\cr
				\bottomrule
			\end{tabular}	
		\end{tabular}
	\end{table}

	It is clear to see that when the solution $u^{3ds5}$ is smooth, both methods work nicely and our collocation method 
is much accurate. 

Next let us consider the non-smooth solution $u^{3ds6}$ in Table~\ref{table3dcompare3}. 
Table \ref{table3dcompare3} shows numerical results such as $l_2, h_1$ errors of these two methods 
for various mesh sizes. Our method can get more accurate solution with $D=5, r=1$ with the large mesh size $|\triangle|$. However, it is clear that an adaptive method is needed to improve the approximation since the maximal error, $e_s=u-u_s$, is worst  near the point $(1,1,1)$.
\end{example}

\subsection{Numerical Results over Nonconvex Domains}
\label{sec:3.3}
\begin{figure}[htpb]
	\centering
	\begin{tabular}{ccc}
		\includegraphics[width=.33\linewidth]{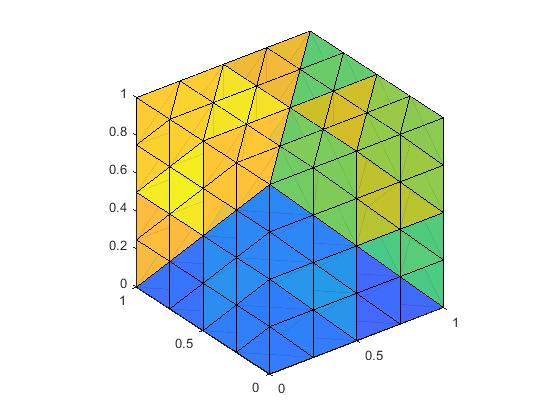}
		&
		\includegraphics[width=.33\linewidth]{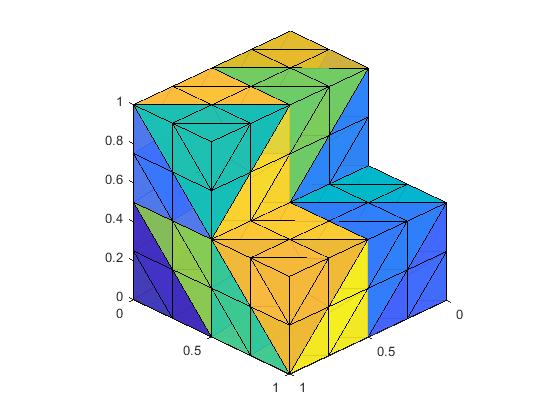}
		&
		\includegraphics[width=.33\linewidth]{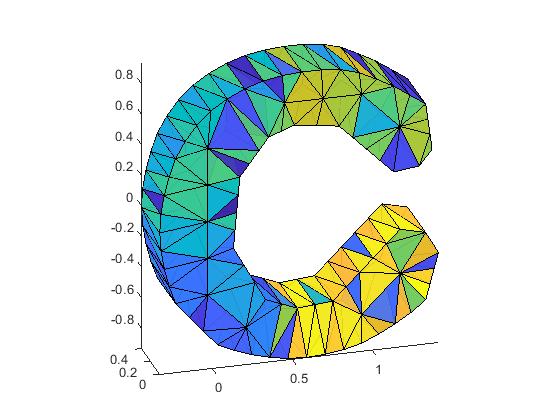}\cr
		\includegraphics[width=.33\linewidth]{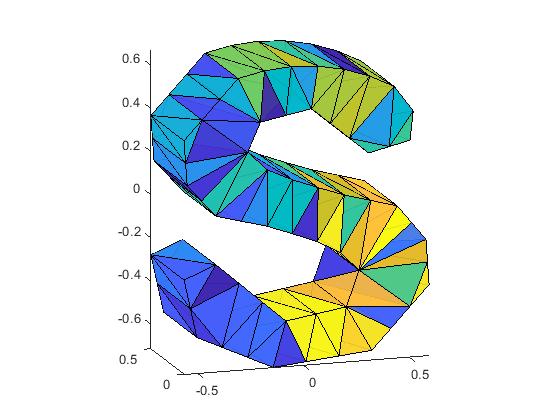}
			&
		\includegraphics[width=.33\linewidth]{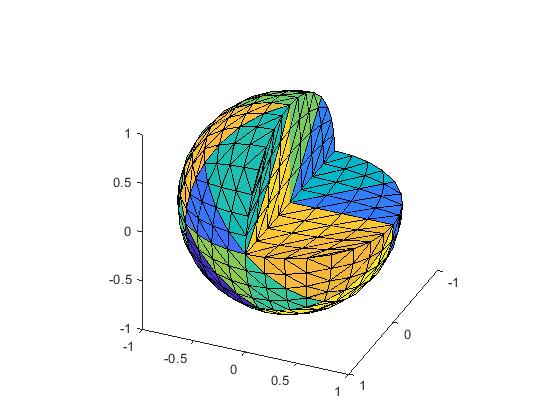}
		\cr
	\end{tabular}
	\caption{Several 3D domains (Top : Cube, Letter L, Letter C , Bottom: Letter S, Subset of the unit ball)}\label{3dfdomain1}
\end{figure}
In this section, we test various solutions for each domain in Figure \ref{3dfdomain1}.
We display CPU times versus the number of vertices and triangles in Table \ref{MATime1} for each domain in Figure \ref{3dfdomain1} when $D=9, r=1$.

\begin{table}[htpb]
	\centering
		\caption{CPU time results and numbers of vertices and tetrahedrons over domains in Figure \ref{3dfdomain1} 
		when $D=9, r=1$}\label{MATime1}
	\begin{tabular}{ c|c c c c} 
	\toprule
		\multicolumn{1}{c|}{Domain}&No. of Vetices&No. of Tetrahedrons&Total CPU(s)  \cr
	\midrule
		Cube &  125& 384&  56.0  \cr  
		Letter L &  105 &   288&  42.7    \cr  
		Letter C& 190&  431&129.2  \cr  
		Letter S & 115&  171& 45.9      \cr  
		\bottomrule
	\end{tabular}
\end{table}

\begin{example}
	We use our method to numerically solve three smooth testing functions $u^{3ds3}, u^{3ds4}, u^{3ds5}$  over 5 solids
which are not strictly convex or not convex.  They  even do not have an uniformly positive reach.   
	Table \ref{table3d3} shows our method performs very well.
	\begin{table}[htpb]
		\centering
			\caption{Errors of numerical solutions $u^{3ds3}-u^{3ds5}$ for Monge Amp\`ere  equations over several domains in Figure \ref{3dfdomain1} for LL methods with $D=9$, $r=1$}\label{table3d3}
		\begin{tabular}{ c c c c c  c c c c}
		\toprule 
			\multirow{2}{1cm}{Solution}&	\multicolumn{2}{c}{Cube}	&	\multicolumn{2}{c}{Letter L}	&	\multicolumn{2}{c}{Letter C}	&	\multicolumn{2}{c}{Letter S}\cr
		\cmidrule(lr){2-3} \cmidrule(r){4-5} \cmidrule(r){6-7} \cmidrule(r){8-9}
			 &$|e_s|_{l_2}$&$|e_s|_{h^1}$&$|e_s|_{l_2}$&$|e_s|_{h^1}$ &$|e_s|_{l_2}$&$|e_s|_{h^1}$&$|e_s|_{l_2}$&$|e_s|_{h^1}$  \\
			\midrule
			$u^{3ds3}$ & 1.76e-09 & 1.64e-07&6.63e-10& 2.58e-08  & 1.48e-08& 8.67e-07 &  3.39e-11& 1.75e-09       \cr  
			$u^{3ds4}$& 2.82e-11 & 1.91e-10 & 3.90e-11& 1.04e-09  & 2.31e-08& 3.56e-07  & 5.84e-11& 2.70e-09       \cr 
			$u^{3ds5}$ & 5.05e-02 & 3.61e-01 &  2.47e-08& 9.56e-07 & 3.87e-08& 3.32e-06   &  6.03e-10& 5.03e-08       \cr  
		\bottomrule
		\end{tabular}
	\end{table}
\end{example}

\begin{example} 
	In \cite{CGG18}, the researchers considered the problem over the unit ball 
	$\Omega=\{(x,y,z)| x^2+y^2+z^2<1\}$ and a convex solution 
	$$u^{3ds7}=-\frac{1}{2\sqrt{3}}(1-x^2-y^2-z^2)$$ 
	of the Monge-Amp\'ere-Dirichlet problem with $f=\frac{1}{3\sqrt{3}}.$ They experimented their numerical solutions
	(called LR method) over the unit ball as well as the 3/4 ball as shown in  Figure \ref{3dfdomain1}. 

	In Table \ref{table3d7}, we first include the numerical results from \cite{CGG18} 
	and then compare the $L^2(\Omega), H^1(\Omega)$ norms of the computed approximation 
	error $u^{3ds7}-u_s$ by our spline collocation method. 
	\begin{table}[htpb]
			\caption{Errors of numerical approximation of solution $u^{3ds7}$ for Monge Amp\`ere  equation 
			over the unit ball for the LR method and the LL method with $D=5$, $r=1$}\label{table3d7}
				\centering
		\begin{tabular}{c c}
			\begin{tabular}{ c c  c c c } 
				\toprule
				\multicolumn{5}{c}{LR method}  \\
				\hline
				\multicolumn{1}{c}{$|\triangle|$}&$|e_s|_{l_2}$&rate&$|e_s|_{h^1}$&rate  \\
			\midrule
				2.98e-01&3.26e-02&-&2.60e-01&-\cr
				1.61e-01&1.11e-02&1.74&1.28e-01&1.14\cr
				8.32e-02&3.22e-03&1.88&6.16e-02&1.11\cr
				4.34e-02&9.89e-04&1.80&2.86e-02&1.17\cr
			\bottomrule
			\end{tabular}
			&
			\begin{tabular}{ c c  c c c } 
				\toprule
				\multicolumn{5}{c}{LL method}  \\
				\hline
				\multicolumn{1}{c}{$|\triangle|$} &$|e_s|_{l_2}$&rate&$|e_s|_{h^1}$&rate  \\
				\midrule
				1 & 3.71e-13  & - & 3.15e-12 & -\cr  
				0.5& 2.97e-14  &3.64 & 1.39e-13 &4.51 \cr  
				\bottomrule
			\end{tabular}	
		\end{tabular}
	\end{table}
	In addition, we tested the solution $u^{3ds7}$ over the subset of the unit ball as shown in Figure \ref{3dfdomain1}. 
	The numerical results we obtained are displayed in Table \ref{table3d8}.
	\begin{table}[htpb]
		\centering
			\caption{CPU time and errors of our spline solution to $u^{3ds7}$ for Monge Amp\`ere  equation over the domain in Figure \ref{3dfdomain1}  with $D=5$, $r=1$, the number of vertices=585, the number of tetrahedrons=2304}\label{table3d8}
		\begin{tabular}{ c  c c} 
			\toprule
			\multicolumn{3}{c}{LL method}  \\
			\hline
		CPU time &$|e_s|_{l_2}$&$|e_s|_{h^1}$ \\
		\midrule
			174.70 &  1.9005e-08 &  2.4941e-06 \cr  
			\hline
		\end{tabular}
	\end{table}
\end{example}
\subsection{Comparison with Numerical Method in \cite{VNNP20}}
\label{sec:3.4}
In this section, we compare our LL method with the Cascadic method in \cite{VNNP20}. The researchers presented several examples in  \cite{VNNP20} over the irregular domains in Figure \ref{domain3d} by using the following test functions
\begin{align*}
	u^{3ds3}&=e^{\frac{(x^2+y^2+z^2)}{2}},\cr
	u^{3ds6}&=-\sqrt{3-(x^2+y^2+z^2)}, \cr
	u^{3ds8}&=\dfrac{x^2+y^2+z^2}{2}-\sin(x)-\sin(y)-\sin(z),\cr 
	u^{3ds9}&=\dfrac{(x^2+y^2+z^2)^{\frac{3}{4}}}{3}.
\end{align*}
We use our method (LL method) to compute numerical solutions based on the same testing  functions over the same testing domains. Our numerical results are shown in Tables \ref{GFDM1}, \ref{GFDM2} and \ref{GFDM3}.
\begin{figure}[htpb]
	\centering
	\begin{tabular}{ccc}
		\includegraphics[width=.33\linewidth]{3ddomain1.jpg}
		&
		\includegraphics[width=.33\linewidth]{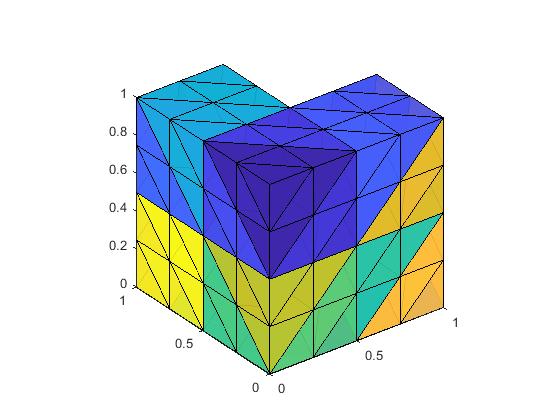}
		&
		\includegraphics[width=.33\linewidth]{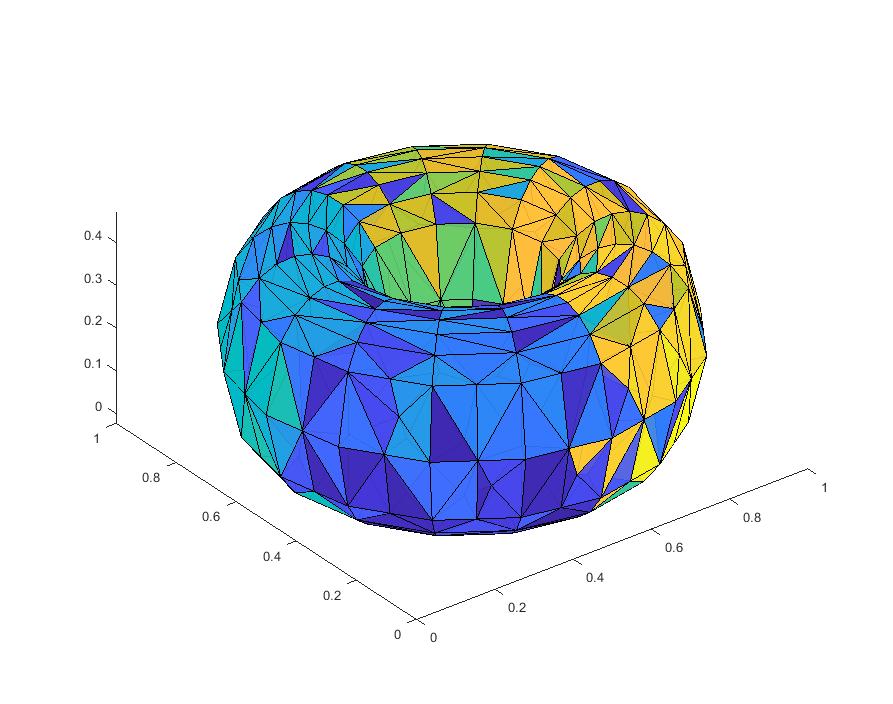}
		\cr
	\end{tabular}
	\caption{Several 3D domains (Cube, Letter L, Torus)}\label{domain3d}
\end{figure}
\begin{table}[htpb]
	\centering
	\caption{The CPU time, DOFs, errors $|e_s|_{l_2}, |e_s|_{h_1}$ using LL method with $D=6, r=1$ and $|e_s|_{l_2}$ using the Cascadic method in \cite{VNNP20} over the cube $[0,1]^3$ }\label{GFDM1}
	\begin{tabular}{ c c  c c c c} 
	\toprule
		&\multicolumn{4}{c}{LL method}&  Cascadic method  \\
	 \cmidrule(lr){2-5} \cmidrule(r){6-6}
		solution&CPU time&DOFs &$|e_s|_{l_2}$&$|e_s|_{h^1}$&$|e_s|_{l_2}$ \\
	\midrule
		$u^{3ds3}$& 18.712 &  32256 &  4.4010e-08 &  1.6144e-06  & 9.8659e-04\cr
		$u^{3ds6}$&6.9873  & 32256 &  4.0285e-04 &  2.6783e-02 &  1.7831e-04\cr
		$u^{3ds8}$&14.852 &  32256 &  1.8152e-11 &  6.8884e-10 & 3.5044e-04\cr
		$u^{3ds9}$&12.826 &  32256  & 1.3242e-04 &  1.1264e-03 &  2.2255e-04\cr
	\bottomrule
	\end{tabular}
\end{table}

\begin{table}[htpb]
	\centering
		\caption{The CPU time, DOFs, errors $|e_s|_{l_2}, |e_s|_{h_1}$ using LL method with $D=6, r=1$ and $|e_s|_{l_2}$ using the Cascadic method in \cite{VNNP20} over L-shaped domain}\label{GFDM2}
	\begin{tabular}{ c c  c c c c} 
		\toprule
		&\multicolumn{4}{c}{LL method}&  Cascadic method  \\
	 \cmidrule(lr){2-5} \cmidrule(r){6-6}
		solution&CPU time&DOFs &$|e_s|_{l_2}$&$|e_s|_{h^1}$&$|e_s|_{l_2}$ \\
		\midrule
		$u^{3ds3}$&  12.826 &  24192 &  1.5727e-08 &  7.0173e-07  &4.8655e-03\cr
		$u^{3ds6}$&4.5050 & 24192  & 6.5992e-05  & 7.6198e-04  &  1.6238e-04\cr
		$u^{3ds8}$&12.935  & 24192  & 1.4826e-11 &  6.9814e-10  &  1.2240e-04\cr
		$u^{3ds9}$&4.8951 & 24192  & 2.3076e-04 &  2.0929e-03 &  4.2183e-04\cr
		\bottomrule
	\end{tabular}
\end{table}

\begin{table}[htpb]
	\centering
		\caption{The CPU time, DOFs, errors $|e_s|_{l_2}, |e_s|_{h_1}$ using LL method with $D=4, r=1$ and $|e_s|_{l_2}$ using the Cascadic method in \cite{VNNP20} over Torus }\label{GFDM3}
	\begin{tabular}{ c c  c c c c} 
		\toprule
		&\multicolumn{4}{c}{LL method}& \multicolumn{1}{c} {Cascadic method}  \\
	 \cmidrule(lr){2-5} \cmidrule(r){6-6}
		solution&CPU time&DOFs &$|e_s|_{l_2}$&$|e_s|_{h^1}$&$|e_s|_{l_2}$ \\
	\midrule
		$u^{3ds3}$&141.01 & 80990 &  6.0125e-07 &  1.3250e-05 &  3.1914e-04\cr
		$u^{3ds6}$&90.420  & 80990 &  5.6340e-04 &  1.1438e-02 &  1.9850e-04\cr
		$u^{3ds8}$&119.42  & 80990 &  4.9540e-07 &  9.5117e-06 & 2.1182e-04\cr
		$u^{3ds9}$&125.64  & 80990 &  2.3272e-07 &  1.3522e-05 &  1.7504e-04\cr
	\bottomrule
	\end{tabular}
\end{table}

\subsection{Free Movement of Transportation}
\label{sec:3.5}
Finally, we consider the free movement problem. 
In this case, the Monge-Amp\'ere equation (\ref{MAE}) becomes
\begin{align}\label{MAE2}
	\det(D^2u(\bfx))&=1, ~~\bfx~\text{in} ~ \Omega \subset \mathbb{R}^3\\
	\nabla u(\bfx)&=\partial W, ~~\bfx ~\text{on} ~\partial \Omega.
\end{align}
To completely determine $u$, we need to specify an oblique boundary condition. When both $\Omega$ and $W$ are star-shaped
domains, e.g., convex domains, we can match the centers of $\Omega$ and $W$ by shifts and use a ray $R$ from the center 
to intercept the boundary of $\Omega$ and the boundary of $W$. This can build up a map from $\partial \Omega$ to the 
boundary of $W$. Using the outward normal of $\partial \Omega$ and the outward normal of $W$, we solve the Neumann 
boundary condition of the Monge-Amp\'ere equation (\ref{MAE}) becomes
\begin{align}\label{MAE3}
	\det(D^2u(\bfx))&=1,~~\bfx~\text{in} ~ \Omega \subset \mathbb{R}^3\\
	\bfn_{\partial \Omega} \nabla u(\bfx)&= \bfn_W\partial W, \bfx ~~\text{on} ~\partial \Omega.
\end{align}

Now we can apply our computational approach to find a solution of $u$ and form a transportation map from $\Omega$ to 
$W$. So that we can show which points in $\Omega$ is mapped to which points in $W$ by using an image or stack of images.

\begin{example}
	For simplicity, we first consider the 2D setting over a rectangular domain $\Omega$ and $W$ is a circular domain as
	shown in Figure~\ref{fig2dex1}. On the left-hand side, we have an image density function and on the right-hand side,
	the density is transported to the circular domain.  The center of the rectangular domain $\Omega
	=[-1, 1]^2$ is the origin $(0,0)$ and
	the same for the circular domain $W=\{(x,y), x^2+y^2\le 1\}$.  
	In Figure~\ref{fig2dex2} and \ref{fig2dex3}, we use the point $(-0.4, 0)$ and $(0,0.4)$ as the center of the circular 
	domain, respectively. The image on the right-hand side of  Figure~\ref{fig2dex2} is clearly deformed and the similar
	for the image (right) of Figure~\ref{fig2dex3}.  The cost for transportation in these two cases  
	\begin{equation}
		\label{cost}
		\frac{1}{2}\int_\Omega \|\bfx - \nabla u(\bfx)\|^2 f(\bfx) d\bfx 
	\end{equation}
	is larger than the cost for the transportation in Figure~\ref{fig2dex1}.  
	
	\begin{figure}[htpb]
		\centering
		\includegraphics[width=1\linewidth, height=4.5cm]{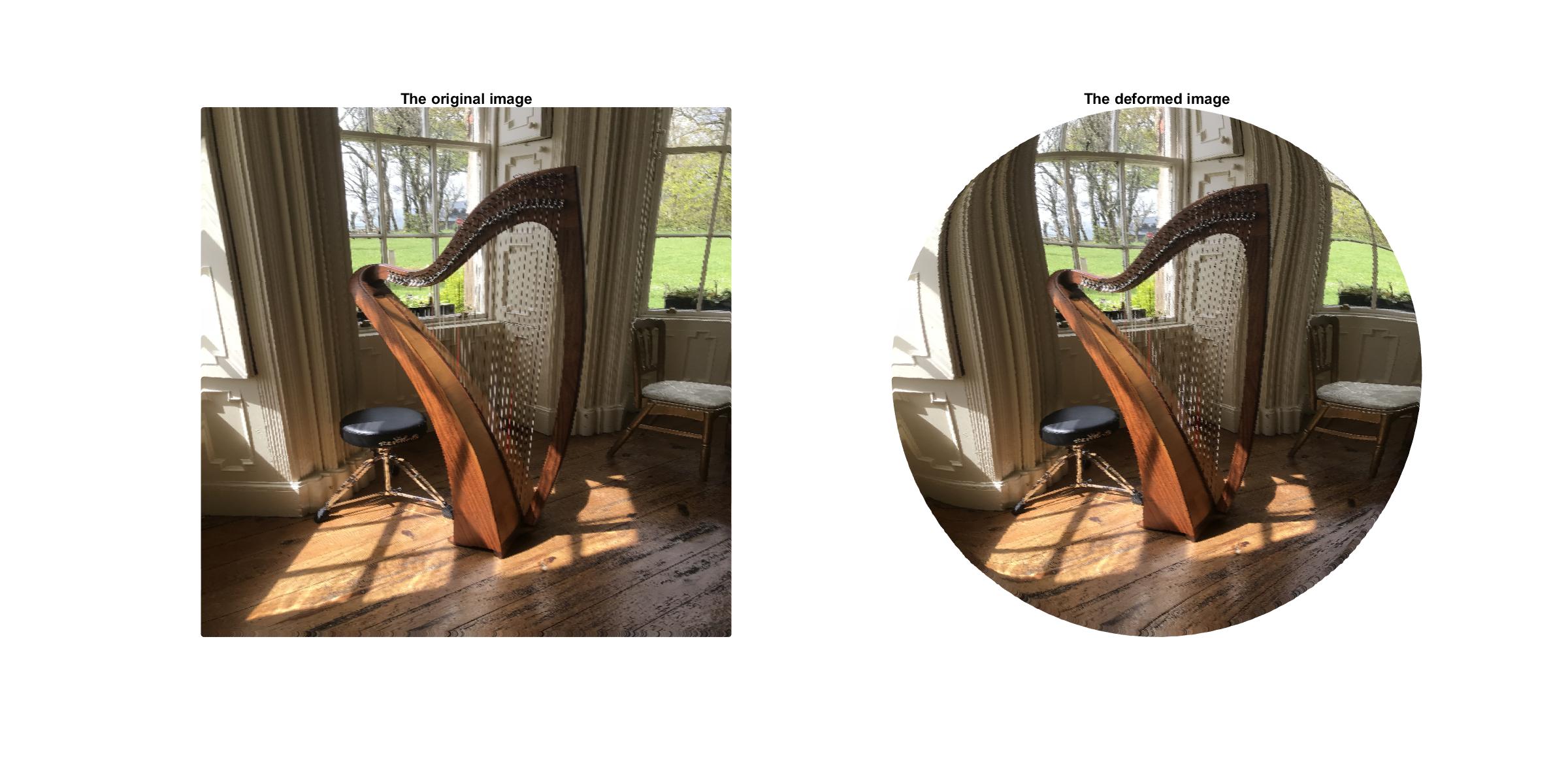}
		\caption{A density over a rectangular domain $\Omega=[-1, 1]^2$(on the left) and a transported image over
			the circular domain (on the right) \label{fig2dex1}}   
	\end{figure}
	
	\begin{figure}[htpb]
		\centering
		\includegraphics[width=1\linewidth, height=4.5cm]{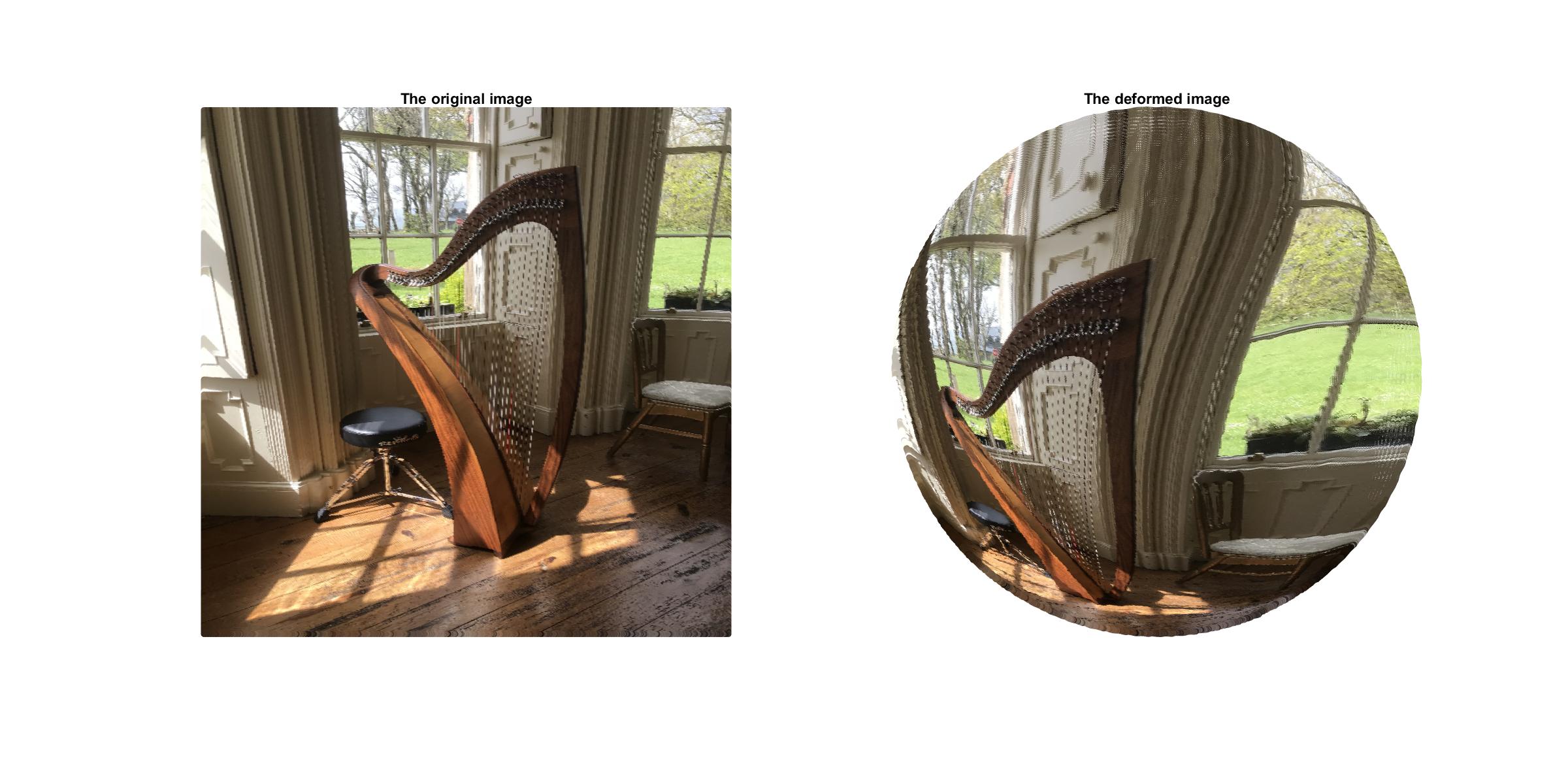}
		\caption{A density over a rectangular domain $\Omega=[-1, 1]^2$(on the left) and a transported image over
			the circular domain (on the right). Note the point $(-0.4,0)$ in the circular domain was chosen as the center. 
			\label{fig2dex2}}
	\end{figure}
	\begin{figure}[htpb]
		\centering
		\includegraphics[width=1\linewidth, height=4.5cm]{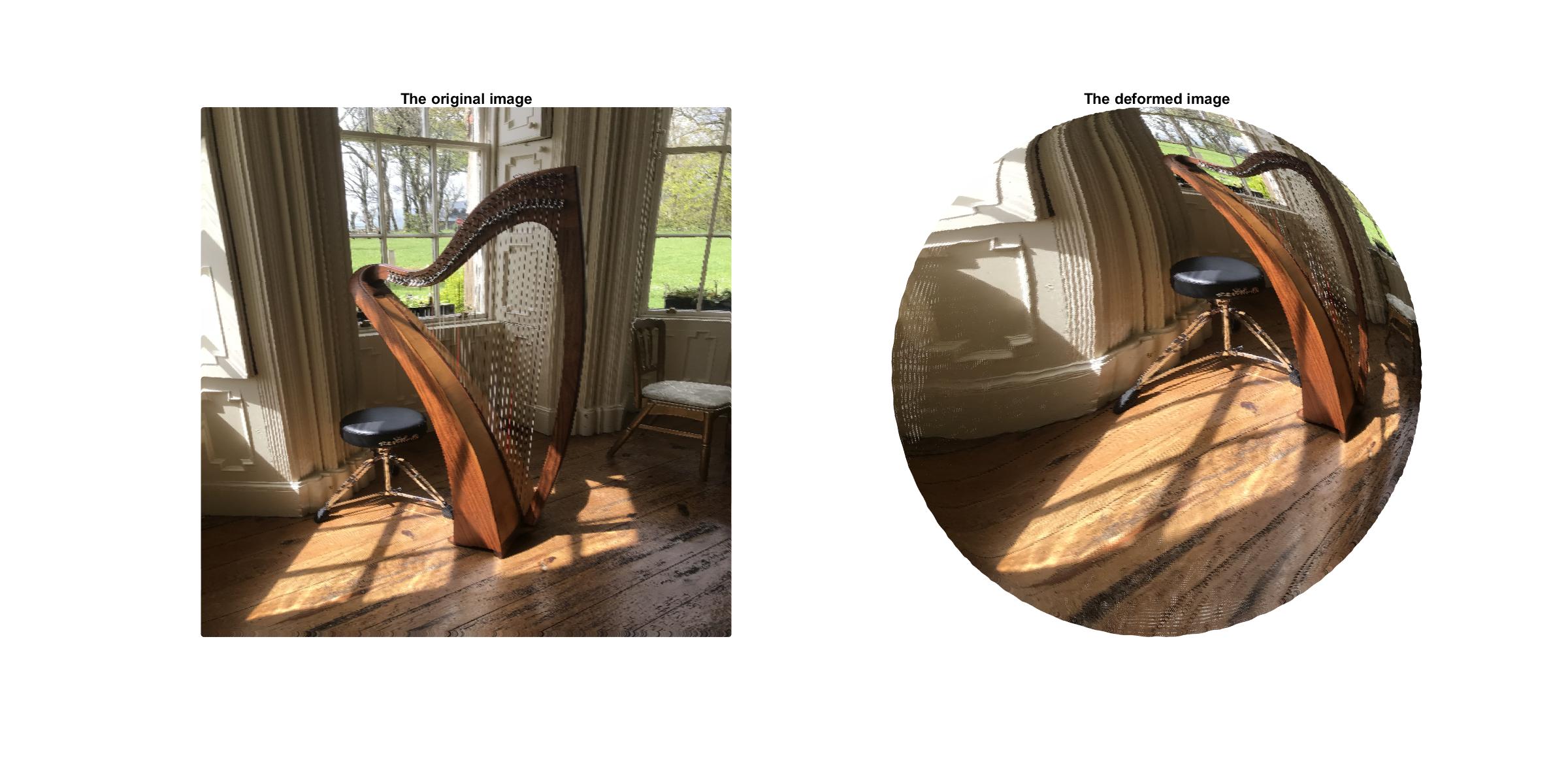}
		\caption{A density over a rectangular domain $\Omega=[-1, 1]^2$(on the left) and a transported image over
			the circular domain (on the right). Note the point $(0,0.4)$ in the circular domain was chosen as the center. 
			\label{fig2dex3}}
	\end{figure}
\end{example}

\begin{example}
	We now show the transportation from the points in the cube $\Omega=[-1, 1]^3$ to the unit ball 
	$W=\{(x,y,z): x^2+y^2+z^2\le 1\}$.
	Again we use the density $f(\bfx)$ which is a stack of the same image over $\Omega$ to show how a point in 
	$\Omega$ is transported to the point in $W$.  
	\begin{figure}[htpb]
		\centering
		\includegraphics[width=1\linewidth, height=6cm]{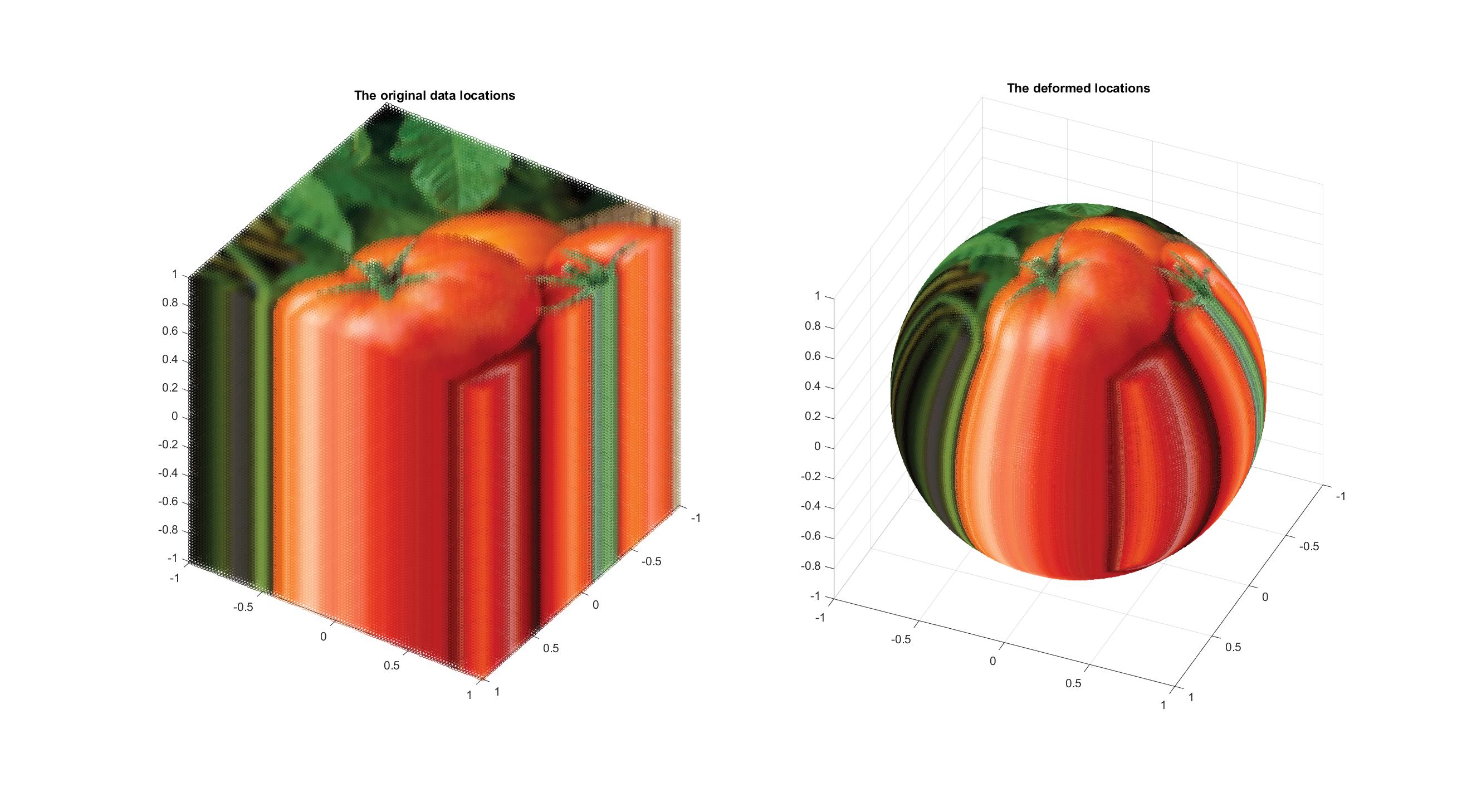}
		\caption{A density over a rectangular domain $\Omega=[-1, 1]^3$(on the left) and a transported image over
			the spherical domain $W$ (on the right) \label{fig3dex1}}   
	\end{figure}
	
	We can see that our computation is reliable as the points in $\Omega$ are completed transported into $W$. In fact
	the map $\nabla u$ is a bijection due to the convexity of the Brenier potential $u$ as shown in Figure~\ref{fig3dex2}.
	\begin{figure}[htpb]
		\centering
		\includegraphics[height=4.5cm]{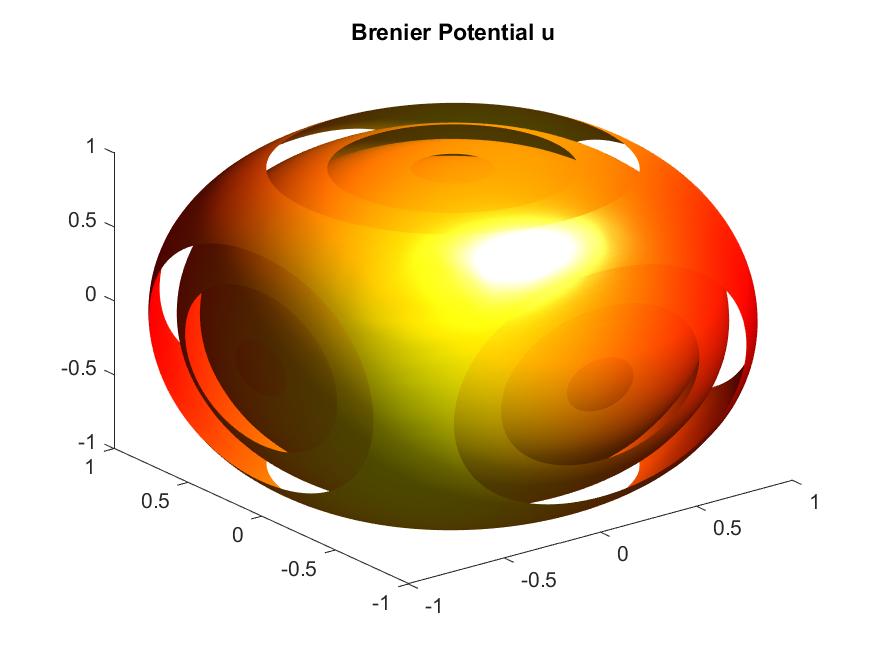}
		\caption{An iso-surface plot of the Brenier potential over a rectangular domain $\Omega=[-1, 1]^3$ 
			\label{fig3dex2}}   
	\end{figure}
\end{example}

\begin{acknowledgements}
The authors would like to thanks anonymous referees for their valuable comments. In addition, the authors 
would like to thank Dr. Gerard Awanou for generosity for providing references \cite{CLW21} and \cite{W96}.   This
work is supported by the Simons Foundation collaboration grant \#864439.
\end{acknowledgements}

%
%



\end{document}